\documentclass[11pt,a4paper]{amsart}

\usepackage{amsmath,amsthm,amssymb}
\usepackage{tikz}
\usepackage{caption}

\usetikzlibrary{arrows,decorations.pathmorphing,decorations.pathreplacing}

\tikzset{vertex/.style={circle,fill=black,inner sep=1pt,outer sep=2pt},
         mvertex/.style={rectangle,draw=black,thick,inner sep=2pt,outer sep=2pt},
         tvertex/.style={inner sep=1pt,font=\scriptsize},
         unvertex/.style={circle,fill=white,draw=white,inner sep=1pt},
         fill1/.style={fill=black!20,draw=black!20},
         fill2/.style={fill=black!40,draw=black!40},
         fill12/.style={fill=black!60,draw=black!60},
         >=stealth',
         leadsto/.style={-angle 90,decorate,decoration=snake,very thick},
         cut/.style={decorate,decoration=saw,very thick}}

\newtheorem{theorem}{Theorem}[section]

\newtheorem{corollary}[theorem]{Corollary}
\newtheorem{lemma}[theorem]{Lemma}
\newtheorem{proposition}[theorem]{Proposition}
\theoremstyle{definition}
\newtheorem{definition}[theorem]{Definition}
\newtheorem{construction}[theorem]{Construction}
\theoremstyle{definition}
\newtheorem{example}[theorem]{Example}
\theoremstyle{remark}
\newtheorem*{remark}{Remark}
\newtheorem*{remarks}{Remarks}
\theoremstyle{definition}
\newtheorem{step}{Step}

\DeclareMathOperator{\Ann}{Ann}
\DeclareMathOperator{\add}{add}
\DeclareMathOperator{\End}{End}
\DeclareMathOperator{\Ext}{Ext}
\DeclareMathOperator{\gldim}{gl.dim}
\DeclareMathOperator{\Hom}{Hom}
\DeclareMathOperator{\id}{id} 
\DeclareMathOperator{\ind}{ind}

\DeclareMathOperator{\rad}{rad}
\DeclareMathOperator{\sspan}{span}

\renewcommand{\mod}{\operatorname{mod}\nolimits}
\newcommand{\op}{{\operatorname{op}\nolimits}}

\newcommand{\orr}{\overset{\tau}{\leadsto}}
\newcommand{\ors}{\overset{\tau}{\to}}
\newcommand{\eqr}{\overset{\tau}{=}}
\newcommand{\cl}[2][]{[#2]_{\eqr_{#1}}}

\newcommand{\replacevertex}[3][fill=white,draw=white]
 {
  \node at #2 [#1,circle,inner sep=1pt] {};
  \node #2 at #2 #3;
 }

\newcommand{\etalchar}[1]{$^{#1}$}

\newcommand{\wB}{ \widetilde{B}}
\newcommand{\wC}{ \widetilde{C}}

\newcommand{\from}{\colon \!}
\newcommand{\wT}{ \widetilde{T}} 

\newcommand{\cA}{ \ensuremath{ {\mathcal A} } }
\newcommand{\cC}{ \ensuremath{ {\mathcal C} } }
\newcommand{\cD}{ \ensuremath{ {\mathcal D} } }

\numberwithin{figure}{section}

\title[Tilted algebras from cluster-tilted algebras]{Constructing tilted algebras from cluster-tilted algebras}
\date{\today}

\author[Bertani-{\O}kland]{Marco Angel Bertani-{\O}kland}
\author[Oppermann]{Steffen Oppermann}
\author[Wr{\aa}lsen]{Anette Wr{\aa}lsen}

\address{Institutt for matematiske fag\\ NTNU\\ 7491 Trondheim\\ Norway}
\email{Marco.Tepetla@math.ntnu.no}
\email{Steffen.Oppermann@math.ntnu.no}
\email{Anette.Wralsen@math.ntnu.no}

\begin{document}

\begin{abstract}
Any cluster-tilted algebra is the relation extension of a tilted algebra. We present a method to, given the distribution of a cluster-tilting object in the Auslander-Reiten quiver of the cluster category, construct all tilted algebras whose relation extension is the endomorphism ring of this cluster-tilting object.
\end{abstract}

\maketitle

\section{Introduction}

The cluster categories of finite dimensional hereditary algebras $H$ were introduced in \cite{BMRRT} in order to give a categorical model to better understand the cluster algebras of Fomin and Zelevinsky \cite{FZ}. The theory of cluster-tilted algebras was initiated in \cite{BMR1}, and the first link from cluster algebras to tilting theory was given in \cite{MRZ}.

There is a close connection between tilted algebras and cluster-tilted algebras (see Section~\ref{section_preliminaries} for definitions and notation). One such connection is the following: From the quiver of a tilted algebra one can obtain the quiver of a cluster-tilted algebra by adding arrows where there are minimal relations (this was proved for some cases in \cite{BR} and \cite{BRS}, and in full generality in \cite{ABS1}). In this paper we explore the opposite problem, i.e.\ to remove arrows from the quiver of a cluster-tilted algebra in such a way that the resulting quiver is the quiver of a tilted algebra.

More precisely, by \cite[1.1]{ABS1} any cluster-tilted algebra is the relation extension of some tilted algebra. Given a cluster-tilted algebra we wish to find all tilted algebras which have the given cluster-tilted algebra as relation extension. We will call these tilted algebras \emph{maximal tilted subalgebras} of the cluster-tilted algebra. For an arbitrary cluster-tilted algebra, given the distribution of a corresponding cluster-tilting object in the Auslander-Reiten quiver of the cluster category, we present an algorithm to construct all maximal tilted subalgebras.

Note that, by \cite{BMR2} and \cite{otherCCS}, all quivers of cluster-tilted algebras are constructed by quiver mutation from acyclic quivers. In the case of a cluster-tilted algebra of finite type, in \cite{BOW} the authors show explicitly how to determine the distribution of the corresponding cluster-tilting object in the cluster category. So in this way we can construct the input of our algorithm.

Our construction consists of the following two main steps:

 First we use local slices to lift the cluster-tilting object to a tilting complex in the derived category. The theory of local slices was introduced in \cite{ABS2} as a way to decide whether two tilted algebras have the same relation-extension algebra. The maximal tilted subalgebras are precisely the endomorphism rings of these tilting complexes. We show that there are certain equivalence classes of local slices which produce the same maximal tilted subalgebras. Moreover we can move from one equivalence class to another (transitively) by ``jumping trenches'' (see Construction~\ref{construction_jumping}).

 Second we use generalized $2$-APR tilts to keep track of the maximal tilted algebras coming up for the various equivalence classes of local slices. The procedure of $n$-APR tilting was introduced in \cite{IO} as a generalization of APR tilting (see \cite{APR}) in order to generate module categories that have a cluster-tilting  object. For $n=2$ the effect of this operation on the quivers and relations of the algebras is completely understood. Here we generalize $2$-APR tilting to complexes in the derived category, and show that jumping trenches is a special case of this generalization. Hence we obtain control over the quivers and relations of the algebras produced in this way.

The paper is organized as follows:

In Section~\ref{section_preliminaries} we will recall some basic results on mutation of quivers, cluster categories, cluster-tilted algebras of finite type and their relations. 

In Sections~\ref{section.Loc_slices} and \ref{section.2apr} we develop the theory for the two steps described above.

In Section~\ref{section_algorithm} we sum up the algorithm to find all the maximal tilted subalgebras of a given cluster-tilted algebra and illustrate it with an example.

Finally, in Section~\ref{section_infinite} we sketch how to apply the algorithm for cluster-tilted algebras of infinite type.

After completing this work we have been informed that similar results have been obtained independently by Bordino, Fern\'{a}ndez, and Trepode (\cite{BFT}).

\section{Background} \label{section_preliminaries}

\subsection{Quiver mutation} Let $Q$ be a finite quiver with no loops or 2-cycles and $k$ a vertex. To mutate at the vertex $k$ and obtain the quiver $\mu_k(Q)$ we do the following. 
\begin{itemize}
 \item Suppose there are $r>0$ arrows $i \to k $, $s>0$ arrows $k \to j$ and $t$ arrows $j \to i$ in $Q$, where a negative number of arrows means arrows in the opposite direction. Then there are $r$ arrows $k \to i$, $s$ arrows $j \to k$ and $t - rs$ arrows $j \to i$ in $\mu_k(Q)$.  
\item All other arrows are kept the same.
\end{itemize}

We say that $Q$ and $\mu_k(Q)$ are \emph{mutation equivalent}. Observe that $\mu^2_k(Q) = Q$. The collection of all quivers that are mutation equivalent to $Q$ is called the mutation class of $Q$. It can be easily seen that this definition is a special case  of matrix mutation, as it appears in the definition of cluster algebras (\cite{FZ}). 

 \subsection{Cluster categories and cluster-tilted algebras} 
 Let $\mathbb{K}$ be an algebraically closed field and $H$ a connected hereditary finite dimensional $\mathbb{K}$-algebra (which we will only call hereditary algebra for the rest of the paper). Any such algebra $H$ is Morita equivalent to a path algebra $\mathbb{K}Q$ for some finite quiver $Q$. An $H$-module $T$ is called a {\em tilting module} if satisfies the following two requirements: $\Ext^1_H(T,T)=0$ and the number of non-isomorphic indecomposable direct summands of $T$ is equal to the number of non-isomorphic simple modules in $\mod H$. The endomorphism algebra $\End_H(T)^\op$ is called a {\em tilted algebra} (see \cite{HR} for further details).

 Let $\cD=D^b(\mod H)$ be the bounded derived category. It comes equipped with two automorphisms, the shift functor $[1]\from \cD \to \cD$ and the Nakayama functor $\nu= - \otimes^L_H DH$ where $D$ denotes the duality on $\mod H$ with respect to the base field $\mathbb{K}$ (see \cite{Happel}). Then one defines the Auslander-Reiten translation $\tau = \nu [-1] \from \cD\to \cD$. Consider the automorphism $F=\tau^{-1}[1]$ of $\cD$ and define the cluster category $\cC=\cC_H$ as the orbit category $\cD/F$. The objects of $\cC$ are the objects of $\cD$, while $\Hom_\cC(A,B)=\oplus_i \Hom_\cD(A,F^i B)$ (see \cite{BMRRT} for more details).
 
 An object $T$ of $\cC$ is called a \emph{(cluster-) tilting} object if $\Ext^1_\cC(T,T)=0$ and $T$ is maximal with respect to this property, i.e.\ if $\Ext^1_\cC(T\oplus X,T\oplus X)=0$, then $X$ is a direct summand of a direct sum of copies of $T$. The endomorphism algebra $\End_\cC(T)^{\op}$ of a tilting object $T$ is called a \emph{cluster-tilted} algebra.
 
 Let $B=\End_\cC(T)^{\op}$ be a cluster-tilted algebra with $\cC=\cC_H$ the cluster category of some hereditary algebra $H$, and $T$ a tilting object in $\cC$.
  We then have  that $B$ is of finite representation type if and only if $H$ is of finite representation type \cite{BMR1}. In this case $H$ is the path algebra of a Dynkin quiver $Q$, and the underlying graph $\Delta$ of $Q$ is one of $\{A_n,D_m,E_6,E_7,E_8\}$ for $n\ge1$ and $m \ge 4$.  We say that $B$ is cluster-tilted of type $\Delta$. 

 We now present a useful theorem from \cite{BMR1}. 
\begin{theorem}[{\cite[2.2]{BMR1}}] \label{theorem.modct}
  Let $T$ be a tilting object in $\cC$. The functor $\Hom_\cC(T,-) \from \cC \to \mod \End_\cC(T)^{\op}$ induces an equivalence $ \cC/\add(\tau T) \simeq \mod \End_\cC(T)^{\op}$. This functor commutes with the AR-translate in both categories and sends AR-triangles to AR-sequences.
\end{theorem}

\subsection{Cluster-tilted algebras and trivial extensions}  
Let $C$ be a finite dimensional algebra of global dimension at most two and consider the $C-C$-bimodule $\Ext^2_C(DC,C)$. We call the trivial extension $C \ltimes \Ext^2_C(DC,C)$ the \emph{relation-extension} of $C$. This definition plays a very important role in the theory of cluster-tilted algebras, as the following theorem shows.

 \begin{theorem}[{\cite[3.4]{ABS1}}] \label{thm.cta_vs_rel_extension}
 An algebra $B$ is cluster-tilted if and only if there exists a tilted algebra $C$ such that $B$ is the relation-extension of $C$.
 \end{theorem}
 Let $B$ be a cluster-tilted algebra.  From \cite[3.3]{BMRRT} we know that there exists a hereditary algebra $H$ and a tilting $H$-module $T'$ such that $B=\End_\cC(T)^{\op}$, where $\cC$ is the cluster category of $H$ and $T$ is the tilting object induced by $T'$, i.e.\ $T$ is the image of $T'$ under the natural embedding $i \from \mod H \to \cC$. Consider now the tilted algebra $C = \End_H(T')^{\op}$. Then we have that $B \simeq \End_H(T')^{\op} \ltimes \Hom_{\mathcal{D}}(T',FT')$ (\cite[proof of 3.1]{Z}). Now for the proof of the Theorem~\ref{thm.cta_vs_rel_extension} above, observe that $\Ext^2_C(DC,C) \simeq \Hom_{\cD}(T',FT')$.
 
 Let $S$ be a subset of the arrows of $Q_B$, the quiver $B$. As in \cite{BRS} we call the set $S$ \emph{admissible}\footnote{Called admissible cut in \cite{BFPPT}.} if $S$ contains exactly one arrow from each full oriented cycle, and no other arrows. Recall that  an oriented cycle in a quiver is called \emph{full} if there are no repeated vertices and  if the subquiver generated by the cycle contains no further arrows.

\section{Cluster-tilted algebras and local slices}
\label{section.Loc_slices}
In this section we discuss the theory of local slices, which lies behind our procedure to find all the maximal tilted subalgebras $C$ of a given cluster-tilted algebra $B$, where maximal means that $C \ltimes \Ext^2(DC,C) = B$. 

 Let $Q=(Q_0,Q_1)$ be a quiver without oriented cycles and $\cC = \cC_Q$ its cluster category. We fix a cluster-tilting object $T = \oplus_a T_a$ in $\cC$, where each $T_a$ is indecomposable for every $a \in Q_0$. Then we have the cluster-tilted algebra $B = \End_\cC (T)^{\op}$ and the induced decomposition $B=\oplus_a B_a$ in indecomposable projective $B$-modules.

Recall that a path $x=x_0 \to x_1 \to \ldots \to x_t=y$ in $\Gamma_\cC$ is \emph{sectional} if, for each $i$ with $0<i<t$, we have $\tau x_{i+1}\neq x_{i-1}$. 

\begin{definition} A \emph{local slice} in
$\cC$ is a full subquiver $\Sigma$ of $\Gamma_\cC$ such that:
\begin{enumerate}
  \item if $x \in \Sigma_0$ and $x \to y$ is an arrow, then either $y
    \in \Sigma_0$ or $\tau y \in \Sigma_0$.  
  \item if $y \in \Sigma_0$ and $x \to y$ is an arrow, then either $x
    \in \Sigma_0$ or $\tau^{-1} x \in \Sigma_0$.  
  \item $\Sigma$ is sectionally convex, i.e.\ if $x = x_0 \to x_1 \to
    \ldots \to x_t = y$ is a sectional path in $\Gamma_{\cC}$, such that
    $x,y \in \Sigma_0$, then $x_i \in \Sigma_0$ for all $i$.
   \item $|\Sigma_{0}|= |Q_0|$.  
\end{enumerate}
By abuse of notation we will sometimes view $\Sigma$ as a set of indecomposable objects, and sometimes as the subcategory consisting of all finite direct sums of these indecomposables.
\end{definition}

\begin{remark}
 Let $\Sigma$ be a local slice in $\cC$ and $T$ a cluster-tilting object such that 
 $\Sigma \cap \add_\cC (\tau T) = 0$. In this case, we say that $\Sigma$ is a local slice in $\cC \setminus \add(\tau T)$. Then, if $\pi \from \cC \to \mod B$ is the projection functor, we have that $\pi(\Sigma)$ is a local slice in $\mod B$ in the sense of \cite[11]{ABS2}. On the other hand, if $\Sigma'$ is a local slice in $\mod B$, then $\pi^{-1}(\Sigma')$ is in $\add (\Sigma \oplus \tau T)$, where $\Sigma$ is a local slice in $\cC$. It is not hard to see that we have a bijection between the local slices in $\cC \setminus \add(\tau T)$ and the set of local slices in $\mod B$. We will identify the two. \end{remark}

 For the rest of this section, we assume the quiver $Q$ to be Dynkin. In this case, we can read off the morphism and extension spaces of the indecomposable objects from the AR-quiver $\Gamma_\cC$. Furthermore, we can explicitly calculate the distribution of $T$ in $\Gamma_\cC$ by using the methods developed in \cite{BOW}. Hence we also assume this distribution to be known. It is therefore easier to illustrate the theory in this case. Later in \S~\ref{section_infinite}, we will explain how to generalize the theory for the infinite case.

We now recall some results from \cite{ABS2} which will be useful for our purposes.
\begin{theorem} \cite[19]{ABS2} \label{theorem.localslices}
 Let  $C$ be a subalgebra of the cluster-tilted algebra $B$. The algebra $C$ is maximal tilted if and only if there exists a local slice $\Sigma$ in $\mod B$ such that $C = B /\Ann_B \Sigma$.
\end{theorem}  
\begin{corollary}\cite[20]{ABS2}
Let $C$ be a tilted algebra and $B$ its relation extension. Then any complete slice in $\mod C$ embeds as a local slice in $\mod B$, and any local slice in $\mod B$ arises this way. 
\end{corollary}
 
 Given a local slice $\Sigma$ in $\mod B$, the ideal $\Ann_B \Sigma$ is generated by a subset $S$ of the set of arrows of the quiver of $B$ (\cite[21]{ABS2}). In fact, $\Ann_B \Sigma \simeq \Ext^2_C(DC,C)$, where $C = B / \Ann_B \Sigma$. We call this admissible set $S$ a \emph{tilted admissible} set. Observe that the arrows that belong to $S$ are obtained from the oriented cycles of $Q_B$. From \cite[3.7]{BR} we know that each of these arrows belongs to exactly one full oriented cycle of $Q_B$. It follows from Theorem~\ref{theorem.localslices} that we have a bijection between the tilted admissible subsets of the set of arrows of $Q_B$ and the maximal tilted subalgebras of $B$.

We want to give a procedure for finding these tilted admissible subsets.
 
\begin{definition}
 Let $X,Y$ be objects in a triangulated category $\cA$. Define
 $I(X,Y)$ to be the set of all the indecomposable objects $Z$ 
 in $\cA$ such that there exist
 morphisms $X \to Z \to Y$ with non-zero composition. 
\end{definition}

These sets of objects will be very useful in order to compute the generating arrows of $\Ann_B \Sigma$, where $\Sigma$ is local slice, by using the following theorem. 

\begin{theorem}\label{theorem.generate_annihilator}
Let $b \to a$ be an arrow of $Q_B$, the quiver of $B=\End_\cC(T)^{\op}$, where $T$ is a cluster-tilting object in $\cC$. Let $\Sigma$ be a local slice in $\mod
B$ and $S$ the tilted admissible set generating $\Ann_B \Sigma$. Then
we have the following:  
\begin{enumerate}
  \item The set $\tau I(T_a,T_b) \setminus \{\tau T_a, \tau T_b\} \neq \emptyset$ if and only if $b \to
    a$ lies on an oriented cycle. 
  \item The arrow  $b\to a $ belongs to $S$ if and only if $\tau
    I(T_a,T_b) \cap \Sigma \neq \emptyset$.
\end{enumerate}
\end{theorem}

\begin{proof} \mbox{ }
\begin{enumerate}
  \item First assume that $b\to a$ lies on an oriented cycle. It is enough to show that $I(T_a,T_b) \setminus \{T_a,T_b\}\neq \emptyset$. Recall that the relations of $B$ are given by a potential (\cite[5.11]{BIRSm},\cite[6.12]{K}).
Thus the arrow $b\to a$ belongs to at least one term in the potential. Choose from one of these terms, a path $\rho$ from $a$ to $b$. Thus we have an associated oriented cycle and we proceed by induction on the length $l$ of the cycle. For $l = 3$ we have the following diagrams in the quiver and in $\cC$:
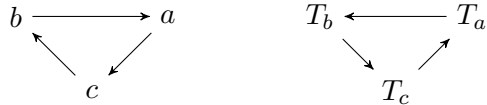
\[ \begin{tikzpicture}[yscale=-1]
 \node (b) at (0,0) {$b$};
 \node (a) at (2,0) {$a$};
 \node (c) at (1,1) {$c$};
 \node (Ta) at (4,0) {$T_b$};
 \node (Tb) at (6,0) {$T_a$};
 \node (Tc) at (5,1) {$T_c$};
 \draw [->] (b) -- (a);
 \draw [->] (a) -- (c);
 \draw [->] (c) -- (b);
 \draw [<-] (Ta) -- (Tb);
 \draw [<-] (Tb) -- (Tc);
 \draw [<-] (Tc) -- (Ta);
\end{tikzpicture} \]
Now mutate at c to obtain:
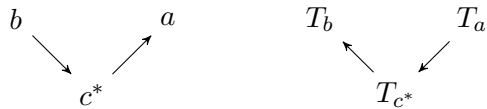
\[ \begin{tikzpicture}[yscale=-1]
 \node (b) at (0,0) {$b$};
 \node (a) at (2,0) {$a$};
 \node (c) at (1,1) {$c^*$};
 \node (Ta) at (4,0) {$T_b$};
 \node (Tb) at (6,0) {$T_a$};
 \node (Tc) at (5,1) {$T_{c^*}$};
 \draw [->] (b) -- (c);
 \draw [->] (c) -- (a);
 \draw [<-] (Ta) -- (Tc);
 \draw [<-] (Tc) -- (Tb);
\end{tikzpicture} \]
Thus $T_{c^*}$ is in $I(T_a,T_b)\setminus \{T_a, T_b\}$ since the composition $T_a \to T_{c^*} \to T_b$ is non-zero. Now for a cycle $b \to a \to c_1 \to \cdots \to c_m \to b$, mutate at $c_m$ to shorten the length of the oriented cycle by one and just restrict to the new oriented cycle as in the following diagram
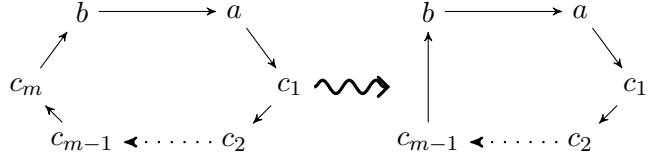
\[ \begin{tikzpicture}[yscale=-1]
 \node at (-4,0) {};
 \node (b) at (-1,0) {$b$};
 \node (a) at (1,0) {$a$};
 \node (c1) at (30:2) {$c_1$};
 \node (c2) at (60:2) {$c_2$};
 \node (c3) at (120:2) {$c_{m-1}$};
 \node (c4) at (150:2) {$c_m$};
 \draw [->] (b) -- (a);
 \draw [->] (a) -- (c1);
 \draw [->] (c1) -- (c2);
 \draw [->,thick,loosely dotted] (c2) -- (c3);
 \draw [->] (c3) -- (c4);
 \draw [->] (c4) -- (b);
\end{tikzpicture}
\begin{tikzpicture}[yscale=-1]
 \draw [leadsto] (-2.5,1) -- (-1.5,1);
 \node (b) at (-1,0) {$b$};
 \node (a) at (1,0) {$a$};
 \node (c1) at (30:2) {$c_1$};
 \node (c2) at (60:2) {$c_2$};
 \node (c3) at (120:2) {$c_{m-1}$};
 \draw [->] (b) -- (a);
 \draw [->] (a) -- (c1);
 \draw [->] (c1) -- (c2);
 \draw [->,thick,loosely dotted] (c2) -- (c3);
 \draw [->] (c3) -- (b);
\end{tikzpicture} \]
        Repeat the procedure until you get to the first case.

    Now assume that there exists $0\ne \tau X \in \tau I(T_a,T_b)\setminus \{\tau T_a, \tau T_b\}$. Resolve $X$ in terms of $T$ to obtain a triangle $X \to T_0 \to T_1 \to X[1]$ where $T_0,T_1 \in \add T$. Let $f \from T_a \to T_b$ be the morphism corresponding to the arrow $b\to a$. The claim follows if we show that there exists a minimal relation involving $f$. Using the triangle above, and the fact that $f$ factors through $X$, we have the following commutative diagram
\[\scalebox{0.9}{
\begin{tikzpicture}[scale=2.5,yscale=-1]
 \node (Ta) at (0,-0.75) {$T_a$};
 \node (T0) at (1,0) {$T_b\oplus T'_0$};
 \node (T1) at (2.25,0) {$T_c\oplus T'_1$};
 \node (X) at (0,0) {$X$};
\draw[->] (Ta) -- (X);
\draw[->] (X) -- (T0);
\draw[->] (Ta) -- node [auto] {$\begin{pmatrix} f \\ g \end{pmatrix}$} (T0.north); 
\draw[->] (T0) -- node [auto] {$\begin{pmatrix} p & q \\ r & s \end{pmatrix}$} (T1); 
\end{tikzpicture}} 
\]
where $p\ne 0$, $T_0=T_b\oplus T'_0$ and $T_1=T_c\oplus T'_1$, for an indecomposable summand $T_c$ of $T$. Since the composition $\left(\begin{smallmatrix} p & q \\ r & s \end{smallmatrix}\right) \left(\begin{smallmatrix} f\\ g \end{smallmatrix}\right)$ is zero, we obtain that $pf + qg=0$. Note that there is no term of the form $pf$ appearing in $qg$, because the approximations of the triangle are minimal.  Hence we have a minimal relation $pf+qg=0$.

  \item Assume $b \to a$ belongs to $S$. Since $S$ generates $\Ann_B \Sigma$ this is equivalent to $b \to a \in \Ann_B \Sigma$. We call the map $T_a \to T_b$ corresponding to this arrow $f$. Then $b \to a \in \Ann_B \Sigma$ if and only if $\Hom_\cC(f, \Sigma) = 0$.

By (the opposite version of) Theorem~\ref{theorem.modct} applied to the cluster-tilting object $\Sigma$ we have an equivalence
\[ \Hom_\cC(-, \Sigma) \colon \cC / (\tau^- \Sigma) \to \mod \End_\cC(\Sigma). \]
In particular $\Hom_\cC(f, \Sigma) = 0$ if and only if $f$ factors through some object in $\tau^- \Sigma$, say through $X$. Then clearly $\tau X \in \tau I (T_a, T_b) \cap \Sigma$, and hence the set is non-empty. If $\tau I (T_a, T_b) \cap \Sigma \neq \emptyset$ the map $\tau f$ factors through $\Sigma$, and thus $f$ factors through $\tau^- \Sigma$.
\end{enumerate}
\end{proof}

 We now give an example illustrating that, in order to produce tilted algebras, the arrows belonging to a tilted admissible set can not be chosen at random.

\begin{example} \label{main.example}
Let $B$ be the cluster-tilted algebra obtained from $D_5$ shown below, and $
C$ the subalgebra of $B$  obtained by removing  $S=\{1\to 2, 3\to 4\}$.
\[ \begin{tikzpicture}[scale=.7,yscale=-1]
 \node (B1) at (3,1) [inner sep=1pt] {1};
 \node (B2) at (2,0) [inner sep=1pt] {2};
 \node (B3) at (2,2) [inner sep=1pt] {3};
 \node (B4) at (1,1) [inner sep=1pt] {4};
 \node (B5) at (0,2) [inner sep=1pt] {5};
 \node (B) at (0,0) {$B \colon$};
 \node (C1) at (8,1) [inner sep=1pt] {1};
 \node (C2) at (7,0) [inner sep=1pt] {2};
 \node (C3) at (7,2) [inner sep=1pt] {3};
 \node (C4) at (6,1) [inner sep=1pt] {4};
 \node (C5) at (5,2) [inner sep=1pt] {5};
 \node (C) at (5,0) {$C \colon$};
 \draw [->] (B4) -- (B5);
 \draw [->] (B4) -- (B1);
 \draw [->] (B1) -- (B2);
 \draw [->] (B1) -- (B3);
 \draw [->] (B2) -- (B4);
 \draw [->] (B3) -- (B4);
 \draw [->] (C4) -- (C5);
 \draw [->] (C4) -- (C1);
 \draw [dashed] (C1) -- (C2);
 \draw [->] (C1) -- (C3);
 \draw [->] (C2) -- (C4);
 \draw [dashed] (C3) -- (C4);
\end{tikzpicture} \]

Here $C$ is not tilted. In fact $\gldim C = 3$, and $C$ is iterated tilted of type $A_5$.
\end{example}

In the light of the previous example, we have the following
definitions. Let $b \to a, c\to d$ be in $S$ where $S$ is an admissible
set in $Q_B$. We say that $b\to a$ and $c\to d$  are
\emph{compatible} if there exists a local slice $\Sigma$ in $\cC \setminus \add(\tau T)$
such that $\Sigma \cap \tau I(T_a,T_b)$ and $ \Sigma \cap \tau I(T_c,T_d)$ are
both non-empty. Otherwise we say that the arrows are \emph{not
  compatible}. The \emph{span} of $b \to a$ is defined to be the
set of indecomposable  modules $X$ in $\cC $ such that there exists
a local slice $\Sigma$  in $\cC \setminus \add(\tau T)$, with $X \in \Sigma$ and
$\Sigma\cap\tau I(T_a,T_b) \neq \emptyset$. We denote it by $\sspan
(b\to a)$. Denote by $\sspan(S) = \cap_{b\to a \text{ in } S} \sspan(b\to
a)$. Thus we have the following. 
\begin{proposition} \label{prop.tilted_admissible.loc_slice} Let $B=\End_\cC (T)^\op$ be a cluster-tilted algebra. 
An admissible set $S$ is tilted if and only if there exists a
local slice $\Sigma \in \cC\setminus \add(\tau T)$ contained in $\sspan(S)$.
\end{proposition}
\begin{proof}
Assume that $S$ is tilted. Then there exists a local slice $\Sigma$ such that $\Ann_B \Sigma$ is generated by $S$. Let $b\to a$ be in $S$. By Theorem~\ref{theorem.generate_annihilator}~(b) we have that $\Sigma \cap \tau I(T_a,T_b)\neq \emptyset$, and thus $\Sigma \subset \sspan(b\to a)$. Since this is true for every arrow of $S$, we conclude that $\Sigma \subset \sspan(S)$.

Assume now that $\Sigma$ is a local slice in $\cC\setminus \add(\tau T)$ contained in $\sspan(S)$. For every $b \to a$ in $S$ we have that $\Sigma \subset \sspan(b\to a)$ and thus $\Sigma \cap \tau I(T_a,T_b)\neq \emptyset $, by the definition of $\sspan(b\to a)$. Let $S'$ be the generating set of $\Ann_B \Sigma$. Then by Theorem~\ref{theorem.generate_annihilator}~(b) we have that $S \subset S'$, but since both sets are admissible, they must be equal. Thus $S$ is tilted.
\end{proof}

For $\cC = \cC_Q$ and $Q$ Dynkin, the Hom-spaces can easily be read off from the AR-quiver, and it is not difficult to compute the sets $I(X,Y)$ for $X,Y$ objects in $\cC$ and the span of an admissible set $S$. 

\begin{example} Let $B$ be the cluster-tilted algebra of type $D_5$ from Example~\ref{main.example} and consider the admissible sets $S_1 = \{1 \to 2, 3\to 4 \}$ and $S_2=\{2\to 4, 3\to 4\}$. Let us check if they are tilted. We do the calculations in the AR-quiver of the cluster category of $D_5$.
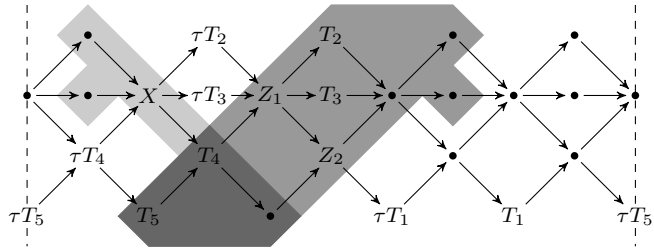
\[ \begin{tikzpicture}[scale=.8,yscale=-1]
 \fill [fill1] (0.5,0) -- (1,0.5) -- (0.5,1) -- (1,1.5) -- (1.5,1) -- (2.5,2) -- (3,1.5) -- (1,-0.5) -- cycle;
 \fill [fill2] (4.5,0) -- (3,1.5) -- (4.5,3) -- (6.5,1) -- (7,1.5) -- (7.5,1) -- (7,0.5) -- (7.5,0) --  (7,-0.5) -- (5, -0.5) -- cycle;
 \fill [fill12] (3,1.5)  -- (2.5,2) -- (1.5,3) -- (2,3.5) -- (4,3.5) -- (4.5,3) -- cycle;
 \node (P31) at (0,1) [vertex] {};
 \node (P51) at (0,3) [tvertex] {$\tau T_5$};
 \node (P11) at (1,0) [vertex]{};
 \node (P21) at (1,1) [vertex]{};
 \node (P41) at (1,2) [tvertex] {$\tau T_4$};
 \node (P32) at (2,1) [tvertex] {$X$};
 \node (P52) at (2,3) [tvertex] {$T_5$};
 \node (P12) at (3,0) [tvertex] {$\tau T_2$};
 \node (P22) at (3,1) [tvertex] {$\tau T_3$};
 \node (P42) at (3,2) [tvertex] {$T_4$};
 \node (P33) at (4,1) [tvertex] {$Z_1$};
 \node (P53) at (4,3) [vertex] {};
 \node (P13) at (5,0) [tvertex] {$T_2$};
 \node (P23) at (5,1) [tvertex] {$T_3$};
 \node (P43) at (5,2) [tvertex] {$Z_2$};
 \node (P34) at (6,1) [vertex] {};
 \node (P54) at (6,3) [tvertex] {$\tau T_1$};
 \node (P14) at (7,0) [vertex] {};
 \node (P24) at (7,1) [vertex] {};
 \node (P44) at (7,2) [vertex] {};
 \node (P35) at (8,1) [vertex] {};
 \node (P55) at (8,3) [tvertex] {$T_1$};
 \node (P15) at (9,0) [vertex] {};
 \node (P25) at (9,1) [vertex] {};
 \node (P45) at (9,2) [vertex] {};
 \node (P36) at (10,1) [vertex] {};
 \node (P56) at (10,3) [tvertex] {$\tau T_5$};
 \draw [dashed] (0,-.5) -- (P51);
 \draw [dashed] (P51) -- (0,3.5);
 \draw [dashed] (10,-.5) -- (P56);
 \draw [dashed] (P56) -- (10,3.5);
 \draw [->] (P31) -- (P11);
 \draw [->] (P31) -- (P21);
 \draw [->] (P31) -- (P41);
 \draw [->] (P51) -- (P41);
 \draw [->] (P11) -- (P32);
 \draw [->] (P21) -- (P32);
 \draw [->] (P41) -- (P32);
 \draw [->] (P41) -- (P52);
 \draw [->] (P32) -- (P12);
 \draw [->] (P32) -- (P22);
 \draw [->] (P32) -- (P42);
 \draw [->] (P52) -- (P42);
 \draw [->] (P12) -- (P33);
 \draw [->] (P22) -- (P33);
 \draw [->] (P42) -- (P33);
 \draw [->] (P42) -- (P53);
 \draw [->] (P33) -- (P13);
 \draw [->] (P33) -- (P23);
 \draw [->] (P33) -- (P43);
 \draw [->] (P53) -- (P43);
 \draw [->] (P13) -- (P34);
 \draw [->] (P23) -- (P34);
 \draw [->] (P43) -- (P34);
 \draw [->] (P43) -- (P54);
 \draw [->] (P34) -- (P14);
 \draw [->] (P34) -- (P24);
 \draw [->] (P34) -- (P44);
 \draw [->] (P54) -- (P44);
 \draw [->] (P14) -- (P35);
 \draw [->] (P24) -- (P35);
 \draw [->] (P44) -- (P35);
 \draw [->] (P44) -- (P55);
 \draw [->] (P35) -- (P15);
 \draw [->] (P35) -- (P25);
 \draw [->] (P35) -- (P45);
 \draw [->] (P55) -- (P45);
 \draw [->] (P15) -- (P36);
 \draw [->] (P25) -- (P36);
 \draw [->] (P45) -- (P36);
 \draw [->] (P45) -- (P56);
\end{tikzpicture} \]
In the figure above, $\tau I(T_4,T_3) = \{X\}$ and $\tau I (T_2,T_1)= \{Z_1,Z_2\}$.  The set $\sspan(3\to 4)$ is shown in light grey, $\sspan(1\to 2)$ in darker grey and the set $\sspan(3\to 4) \cap \sspan(1\to 2)$ in dark grey. It is clear that there is no local slice in the intersection and hence the admissible set $S_1$ is not tilted. Therefore the arrows $3\to 4$ and $1 \to 2$ are not compatible. We already knew that $S_1$ is not tilted, since this admissible set produces the subalgebra $C$ of $B$ in Example~\ref{main.example}. 

Next we consider $S_2$.
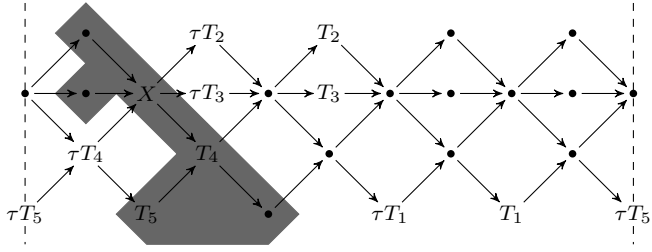
\[ \begin{tikzpicture}[scale=.8,yscale=-1]
 \fill [fill12] (0.5,0) -- (1,0.5) -- (0.5,1) -- (1,1.5) -- (1.5,1) -- (2.5,2) -- (3,1.5) --  (2.5,2) -- (1.5,3) -- (2,3.5) -- (4,3.5) -- (4.5,3) -- (1,-0.5) -- cycle;
 \node (P31) at (0,1) [vertex] {};
 \node (P51) at (0,3) [tvertex] {$\tau T_5$};
 \node (P11) at (1,0) [vertex]{};
 \node (P21) at (1,1) [vertex]{};
 \node (P41) at (1,2) [tvertex] {$\tau T_4$};
 \node (P32) at (2,1) [tvertex] {$X$};
 \node (P52) at (2,3) [tvertex] {$T_5$};
 \node (P12) at (3,0) [tvertex] {$\tau T_2$};
 \node (P22) at (3,1) [tvertex] {$\tau T_3$};
 \node (P42) at (3,2) [tvertex] {$T_4$};
 \node (P33) at (4,1) [vertex] {};
 \node (P53) at (4,3) [vertex] {};
 \node (P13) at (5,0) [tvertex] {$T_2$};
 \node (P23) at (5,1) [tvertex] {$T_3$};
 \node (P43) at (5,2) [vertex] {};
 \node (P34) at (6,1) [vertex] {};
 \node (P54) at (6,3) [tvertex] {$\tau T_1$};
 \node (P14) at (7,0) [vertex] {};
 \node (P24) at (7,1) [vertex] {};
 \node (P44) at (7,2) [vertex] {};
 \node (P35) at (8,1) [vertex] {};
 \node (P55) at (8,3) [tvertex] {$T_1$};
 \node (P15) at (9,0) [vertex] {};
 \node (P25) at (9,1) [vertex] {};
 \node (P45) at (9,2) [vertex] {};
 \node (P36) at (10,1) [vertex] {};
 \node (P56) at (10,3) [tvertex] {$\tau T_5$};
 \draw [dashed] (0,-.5) -- (P51);
 \draw [dashed] (P51) -- (0,3.5);
 \draw [dashed] (10,-.5) -- (P56);
 \draw [dashed] (P56) -- (10,3.5);
 \draw [->] (P31) -- (P11);
 \draw [->] (P31) -- (P21);
 \draw [->] (P31) -- (P41);
 \draw [->] (P51) -- (P41);
 \draw [->] (P11) -- (P32);
 \draw [->] (P21) -- (P32);
 \draw [->] (P41) -- (P32);
 \draw [->] (P41) -- (P52);
 \draw [->] (P32) -- (P12);
 \draw [->] (P32) -- (P22);
 \draw [->] (P32) -- (P42);
 \draw [->] (P52) -- (P42);
 \draw [->] (P12) -- (P33);
 \draw [->] (P22) -- (P33);
 \draw [->] (P42) -- (P33);
 \draw [->] (P42) -- (P53);
 \draw [->] (P33) -- (P13);
 \draw [->] (P33) -- (P23);
 \draw [->] (P33) -- (P43);
 \draw [->] (P53) -- (P43);
 \draw [->] (P13) -- (P34);
 \draw [->] (P23) -- (P34);
 \draw [->] (P43) -- (P34);
 \draw [->] (P43) -- (P54);
 \draw [->] (P34) -- (P14);
 \draw [->] (P34) -- (P24);
 \draw [->] (P34) -- (P44);
 \draw [->] (P54) -- (P44);
 \draw [->] (P14) -- (P35);
 \draw [->] (P24) -- (P35);
 \draw [->] (P44) -- (P35);
 \draw [->] (P44) -- (P55);
 \draw [->] (P35) -- (P15);
 \draw [->] (P35) -- (P25);
 \draw [->] (P35) -- (P45);
 \draw [->] (P55) -- (P45);
 \draw [->] (P15) -- (P36);
 \draw [->] (P25) -- (P36);
 \draw [->] (P45) -- (P36);
 \draw [->] (P45) -- (P56);
\end{tikzpicture} \]
Here we have that $\tau I(T_4,T_3) = \tau I (T_4,T_2)= \{X\}$.  The set $\sspan(3\to 4) = \sspan(2\to 4)$ is shown in dark grey. There are two local slices contained in $\sspan(S_2)$ and thus $S_2$ is tilted and the arrows $3\to 4$ and $2 \to 4$ are compatible. Observe that both local slices share the same annihilator.
\end{example}

As the example above shows, there may be many local slices whose annihilator is generated by the
same tilted admissible set $S$. We will now define an equivalence relation on the set of local slices such that two local slices belong to the same equivalence class if and only if they share the same annihilator.

\begin{definition}
 Let $\Sigma$ be a local slice in $\mod B$ and $X$ an indecomposable  object in $\Sigma$. Define $\tau^{+}_X \Sigma = (\Sigma \setminus{X}) \cup  \tau X$. Similarly, we define $\tau_X^{-}\Sigma =(\Sigma \setminus{X}) \cup \tau^{-} X $. 
\end{definition}

 It is not difficult to see that $\tau^{+}_X \Sigma$ is a local  slice in $\mod B$ if and only if $\tau X$ is defined and $X$ is a sink when restricted to  $\Sigma$. Equivalently, $\tau^{+}_X \Sigma$ is a local slice in $\cC \setminus \add(\tau T)$ if and only if $\tau X \not\in \add(\tau T)$  and  $X$ is a sink when restricted to  $\Sigma$.  There is a dual remark  for $\tau^{-}_X \Sigma$.

\begin{definition} \label{definition.homotopy}
 Let $\Sigma$ and $\Sigma'$ be two local slices in $\mod B$. We write  $\Sigma \sim \Sigma'$ if there exists a  sequence of indecomposable modules $X_1,\ldots,X_m$ such that $\Sigma_i = \tau^{\pm}_{X_i} \Sigma_{i-1}$ is a local slice, $\Sigma_0 = \Sigma$, $\Sigma_m = \Sigma'$ and $X_i \in \Sigma_{i-1}$ for $1 \le i \le  m$. In this case, we say that $\Sigma$ is \emph{homotopic} to $\Sigma'$.
\end{definition}

The symbol $\pm$ means that one can choose either $+$ or $-$ in the  sequence.

 Note that two local slices are homotopic if one can move from one to the other without passing through the ``holes'' of $\mod B$, i.e.\ the holes made by $\tau T$ in the equivalence $\cC/ (\tau T) \simeq \mod B$.

We introduce the following notation for AR-triangles. If $X$ is an indecomposable object in $\cC$, we have two AR-triangles associated to $X$:
\[
  X \to \vartheta^{-}X \to \tau^{-}X \to \mbox{ and }
  \tau X \to \vartheta X \to X \to
\]
where $\vartheta^{-} X$ and $\vartheta X$ just denote the middle term of the corresponding AR-triangle.

We will now define an equivalence relation on the set of indecomposable summands of the cluster-tilting object $T$. 

\begin{definition} \label{definition.celldecomp}
Let $T_a,T_b$ be two non-isomorphic indecomposable summands of $T$. We say that $T_a \equiv_1 T_b$ if there exists an AR-triangle $\tau X \to \vartheta X \to X \to$  such that $ T_a,T_b$ are direct summands of $ \tau X \oplus \vartheta X \oplus X$. Take $\equiv$ to be the minimal equivalence relation containing $\equiv_1$. We call the equivalence class $[T_a]$ a \emph{cell} and $\tau [T_a]$ a \emph{trench}. Then we have a partition of the summands of $T$ and we write $T = \oplus_k \wT_k$, where each $\wT_k$ is the sum of all the indecomposable summands belonging to the same cell. We call this the \emph{cell decomposition} of $T$. 
\end{definition}

Similarly, $B$ and $C$ inherit a cell decomposition, where $C$ is any maximal tilted subalgebra of the cluster-tilted algebra $B$ associated to $T$.

At the level of quivers, we also inherit a cell decomposition. The cells of $Q_B$ are the full subquivers $Q_{\End_\cC (\wT_k)^\op}$ for the corresponding $k$.

 Let $\Sigma$ be a local slice in $\cC\setminus \add(\tau T)$. A cell  $[T_a]$ is called a  \emph{relative source} with respect to $\Sigma$  if whenever there  is a non-zero morphism from the cell $\wT_j$ to $\wT_k$ for $j\neq k$ we have that  $\Sigma\cap\tau I(\wT_j,\wT_k)\neq\emptyset$. Then we also call  the cells $\wB_k$ and $\wC_k$ a relative source. 

\begin{example}\label{example.trenches}
Let $B$ be the cluster-tilted algebra of type $D_5$ from Example~\ref{main.example}. Then a cluster-tilting object $T=\oplus_{i=1}^5 T_i$ such that $B=\End_\cC(
T)^{\op}$ is given in the AR-quiver of the cluster category of  $D_5$ below. 

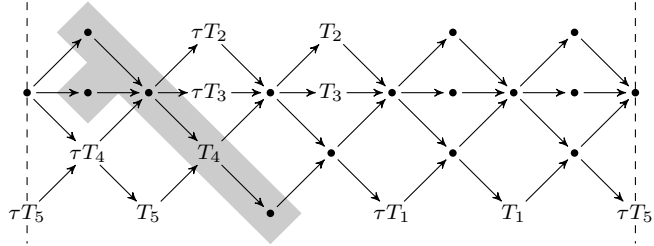
\[ \begin{tikzpicture}[scale=.8,yscale=-1]
 \fill [fill1] (0.5,0) -- (1,0.5) -- (0.5,1) -- (1,1.5) -- (1.5,1) -- (2,1.5) -- (4,3.5) -- (4.5,3) -- (1,-0.5) -- cycle;
 \node (P31) at (0,1) [vertex] {};
 \node (P51) at (0,3) [tvertex] {$\tau T_5$};
 \node (P11) at (1,0) [vertex]{};
 \node (P21) at (1,1) [vertex]{};
 \node (P41) at (1,2) [tvertex] {$\tau T_4$};
 \node (P32) at (2,1) [vertex] {};
 \node (P52) at (2,3) [tvertex] {$T_5$};
 \node (P12) at (3,0) [tvertex] {$\tau T_2$};
 \node (P22) at (3,1) [tvertex] {$\tau T_3$};
 \node (P42) at (3,2) [tvertex] {$T_4$};
 \node (P33) at (4,1) [vertex] {};
 \node (P53) at (4,3) [vertex] {};
 \node (P13) at (5,0) [tvertex] {$T_2$};
 \node (P23) at (5,1) [tvertex] {$T_3$};
 \node (P43) at (5,2) [vertex] {};
 \node (P34) at (6,1) [vertex] {};
 \node (P54) at (6,3) [tvertex] {$\tau T_1$};
 \node (P14) at (7,0) [vertex] {};
 \node (P24) at (7,1) [vertex] {};
 \node (P44) at (7,2) [vertex] {};
 \node (P35) at (8,1) [vertex] {};
 \node (P55) at (8,3) [tvertex] {$T_1$};
 \node (P15) at (9,0) [vertex] {};
 \node (P25) at (9,1) [vertex] {};
 \node (P45) at (9,2) [vertex] {};
 \node (P36) at (10,1) [vertex] {};
 \node (P56) at (10,3) [tvertex] {$\tau T_5$};
 \draw [dashed] (0,-.5) -- (P51);
 \draw [dashed] (P51) -- (0,3.5);
 \draw [dashed] (10,-.5) -- (P56);
 \draw [dashed] (P56) -- (10,3.5);
 \draw [->] (P31) -- (P11);
 \draw [->] (P31) -- (P21);
 \draw [->] (P31) -- (P41);
 \draw [->] (P51) -- (P41);
 \draw [->] (P11) -- (P32);
 \draw [->] (P21) -- (P32);
 \draw [->] (P41) -- (P32);
 \draw [->] (P41) -- (P52);
 \draw [->] (P32) -- (P12);
 \draw [->] (P32) -- (P22);
 \draw [->] (P32) -- (P42);
 \draw [->] (P52) -- (P42);
 \draw [->] (P12) -- (P33);
 \draw [->] (P22) -- (P33);
 \draw [->] (P42) -- (P33);
 \draw [->] (P42) -- (P53);
 \draw [->] (P33) -- (P13);
 \draw [->] (P33) -- (P23);
 \draw [->] (P33) -- (P43);
 \draw [->] (P53) -- (P43);
 \draw [->] (P13) -- (P34);
 \draw [->] (P23) -- (P34);
 \draw [->] (P43) -- (P34);
 \draw [->] (P43) -- (P54);
 \draw [->] (P34) -- (P14);
 \draw [->] (P34) -- (P24);
 \draw [->] (P34) -- (P44);
 \draw [->] (P54) -- (P44);
 \draw [->] (P14) -- (P35);
 \draw [->] (P24) -- (P35);
 \draw [->] (P44) -- (P35);
 \draw [->] (P44) -- (P55);
 \draw [->] (P35) -- (P15);
 \draw [->] (P35) -- (P25);
 \draw [->] (P35) -- (P45);
 \draw [->] (P55) -- (P45);
 \draw [->] (P15) -- (P36);
 \draw [->] (P25) -- (P36);
 \draw [->] (P45) -- (P36);
 \draw [->] (P45) -- (P56);
\end{tikzpicture} \]

Here the dashed lines are identified. We then have three cells, given by $[T_1] = T_1, [T_2] =T_2\oplus T_3$ and $[T_4]=T_4 \oplus T_5$. The three trenches are $X_i = \tau [T_i]$ for $i=1,2,4$.

Let $\Sigma$ be the local slice given by the grey area. Then we have that $[T_2]$ is a relative source with respect to$\Sigma$ and $[T_4]$ is a relative sink with respect to $\Sigma$. The maximal tilted subalgebra $C$ associated to $\Sigma$ is given by
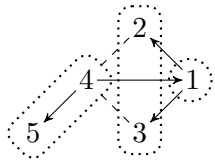
\[ \begin{tikzpicture}[scale=.7,yscale=-1]
 \draw [rounded corners=5pt,dotted,thick] (1.6,-.4) -- (2.4,-.4) -- (2.4,2.4) -- (1.6,2.4) -- cycle;
 \draw [rounded corners=5pt,dotted,thick] (-.57,2) -- (0,2.57) -- (1.57,1) -- (1,.43) -- cycle;
 \draw [dotted,thick](2.95,1) circle (.4);
 \node (C1) at (3,1) [inner sep=1pt] {1};
 \node (C2) at (2,0) [inner sep=1pt] {2};
 \node (C3) at (2,2) [inner sep=1pt] {3};
 \node (C4) at (1,1) [inner sep=1pt] {4};
 \node (C5) at (0,2) [inner sep=1pt] {5};
 \draw [->] (C4) -- (C5);
 \draw [->] (C4) -- (C1);
 \draw [->] (C1) -- (C2);
 \draw [->] (C1) -- (C3);
 \draw [dashed] (C2) -- (C4);
 \draw [dashed] (C3) -- (C4);
\end{tikzpicture} \]
whose cell decomposition is as indicated by the dots in the figure above.
\end{example}

We now give a criterion for when two local slices give rise to the same maximal tilted subalgebra.

\begin{theorem} \label{theorem.homotopy}
Let $\Sigma$ and $\Sigma'$ be two local slices in $\mod B$. Then $\Ann_B \Sigma = \Ann_B \Sigma'$ if and only if $\Sigma \sim \Sigma'$.
\end{theorem}

\begin{proof}
Assume that $\Ann_B \Sigma = \Ann_B \Sigma'$ and that $\Sigma \not\sim \Sigma'$. Two such local slices cut the cluster category into at least two separate parts, each containing part of the trenches. More precisely, there exist trenches $\tau[T_a]$ and $\tau[T_b]$ such that any map between them factors through $\Sigma \oplus \Sigma'$. Let $C = B / \Ann_B \Sigma$ be the tilted algebra corresponding to the local slice $\Sigma$. Since $C$ is connected, we may choose $T_a$ and $T_b$ as above in such a way that there are non-zero homomorphisms between the corresponding indecomposable projective modules of $C$. By Theorem~\ref{theorem.generate_annihilator}~(b) we have $\tau I(T_a, T_b) \cap \Sigma = \emptyset$. Then, by our choice of $[T_a]$ and $[T_b]$, the map $\tau T_a \to \tau T_b$ has to factor through $\Sigma'$. Hence we have found an element of $\Ann_B \Sigma' \setminus \Ann_B \Sigma$, contradicting the assumption.

Now assume that $\Sigma \sim \Sigma'$. Then we can move from one to
the other without passing through any trench. This means that they must
have the same trenches to the left and right, and thus kill the same
arrows from $Q_B$. Hence $\Ann_B \Sigma = \Ann_B \Sigma'$.
\end{proof}

This theorem shows that two local slices produce the same maximal tilted
subalgebra of $B$ if and only if both local slices belong to the same
homotopy class. Hence we have proved the following.

\begin{corollary}
 There is a bijection between the set of homotopy classes of local slices in
 $\mod B$ and the set of maximal tilted subalgebras of $B$.
\end{corollary}

We now want to be able to move from one equivalence class to the other
by ``jumping''  trenches. We will work in $\cC$ since in that category
the trenches are physically there. All our local slices will not
intersect $\add(\tau T)$ and thus will naturally descend to $\mod B$.

\begin{construction} \label{construction_jumping}
Let $X=\tau [T_a]$ be a trench in $\cC$. Define 
\begin{align*}
  L_X = 
& \ind(\tau X \oplus \vartheta X \oplus \vartheta^- \vartheta X)\setminus \ind(\add X) \\
  R_X =
& \ind(\tau^- X \oplus \vartheta^- X \oplus \vartheta \vartheta^-
X)\setminus \ind (\add X) .\\
\end{align*}
\end{construction}
We claim that there exist local slices $\Sigma, \Sigma'$ such that the
only trench between them is $X$. To see this, use $L_X$ and complete it
to a local slice (one can, for instance, use the same algorithm as in the proof of \cite[23]{ABS2}). Now use the same completion with $R_X$. This works
because $R_X$ and $L_X$ intersect at their end-points and their union
surrounds $X$. 

Then we can define the following operations on local slices. For a local slice $\Sigma$ with $L_X \subset \Sigma$ and a local slice $\Sigma'$ with $R_X \subset \Sigma'$ we set
\begin{align*}
J_X^- \Sigma = & (\Sigma \setminus L_X) \cup R_X \\
J_X^+ \Sigma' = & (\Sigma' \setminus R_X) \cup L_X .
\end{align*}
Note that $J_X^+ J_X^- \Sigma = \Sigma$ and $J_X^- J_X^+ \Sigma' =
\Sigma'$. Furthermore, if it is possible to apply $J_X^-$ or $J_X^+$ to two equivalent local slices, then the images will be equivalent again.

One can always choose a representative of each equivalence class of local slices, such that one can apply $J^\pm$.

Notice that $[T_a]$ is a relative source with respect to $\Sigma$ and a
relative sink with respect to $\Sigma'$. Thus $J^-$ transforms relative sources
into relative sinks and $J^+$ does the opposite. It is clear that with this procedure we run through all  the equivalence classes of local slices, and thus through all the maximal tilted subalgebras of $B$.
We now illustrate with an example.

\begin{example} Let $X=\tau [T_2] = \{ \tau T_2, \tau T_3\}$ be the trench as in Example~\ref{example.trenches}.
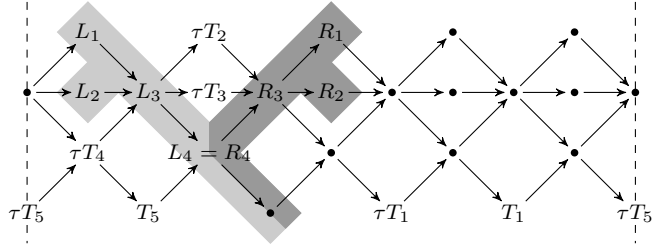
\[ \begin{tikzpicture}[scale=.8,yscale=-1]
 \fill [fill1] (0.5,0) -- (1,0.5) -- (0.5,1) -- (1,1.5) -- (1.5,1) -- (2,1.5) -- (4,3.5) -- (4.5,3) -- (1,-0.5) -- cycle;
 \fill [fill2] (4.5,0) -- (3,1.5) -- (3,2) -- (4.25,3.25) -- (4.5,3) -- (3.5,2) -- (4.5,1) -- (5,1.5)-- (5.5,1) -- (5,0.5) -- (5.5,0) -- (5,-0.5) -- cycle;
 \node (P31) at (0,1) [vertex] {};
 \node (P51) at (0,3) [tvertex] {$\tau T_5$};
 \node (P11) at (1,0) [tvertex] {$L_1$};
 \node (P21) at (1,1) [tvertex] {$L_2$};
 \node (P41) at (1,2) [tvertex] {$\tau T_4$};
 \node (P32) at (2,1) [tvertex] {$L_3$};
 \node (P52) at (2,3) [tvertex] {$T_5$};
 \node (P12) at (3,0) [tvertex] {$\tau T_2$};
 \node (P22) at (3,1) [tvertex] {$\tau T_3$};
 \node (P42) at (3,2) [tvertex] {$L_4=R_4$};
 \node (P33) at (4,1) [tvertex] {$R_3$};
 \node (P53) at (4,3) [vertex] {};
 \node (P13) at (5,0) [tvertex] {$R_1$};
 \node (P23) at (5,1) [tvertex] {$R_2$};
 \node (P43) at (5,2) [vertex] {};
 \node (P34) at (6,1) [vertex] {};
 \node (P54) at (6,3) [tvertex] {$\tau T_1$};
 \node (P14) at (7,0) [vertex] {};
 \node (P24) at (7,1) [vertex] {};
 \node (P44) at (7,2) [vertex] {};
 \node (P35) at (8,1) [vertex] {};
 \node (P55) at (8,3) [tvertex] {$T_1$};
 \node (P15) at (9,0) [vertex] {};
 \node (P25) at (9,1) [vertex] {};
 \node (P45) at (9,2) [vertex] {};
 \node (P36) at (10,1) [vertex] {};
 \node (P56) at (10,3) [tvertex] {$\tau T_5$};
 \draw [dashed] (0,-.5) -- (P51);
 \draw [dashed] (P51) -- (0,3.5);
 \draw [dashed] (10,-.5) -- (P56);
 \draw [dashed] (P56) -- (10,3.5);
 \draw [->] (P31) -- (P11);
 \draw [->] (P31) -- (P21);
 \draw [->] (P31) -- (P41);
 \draw [->] (P51) -- (P41);
 \draw [->] (P11) -- (P32);
 \draw [->] (P21) -- (P32);
 \draw [->] (P41) -- (P32);
 \draw [->] (P41) -- (P52);
 \draw [->] (P32) -- (P12);
 \draw [->] (P32) -- (P22);
 \draw [->] (P32) -- (P42);
 \draw [->] (P52) -- (P42);
 \draw [->] (P12) -- (P33);
 \draw [->] (P22) -- (P33);
 \draw [->] (P42) -- (P33);
 \draw [->] (P42) -- (P53);
 \draw [->] (P33) -- (P13);
 \draw [->] (P33) -- (P23);
 \draw [->] (P33) -- (P43);
 \draw [->] (P53) -- (P43);
 \draw [->] (P13) -- (P34);
 \draw [->] (P23) -- (P34);
 \draw [->] (P43) -- (P34);
 \draw [->] (P43) -- (P54);
 \draw [->] (P34) -- (P14);
 \draw [->] (P34) -- (P24);
 \draw [->] (P34) -- (P44);
 \draw [->] (P54) -- (P44);
 \draw [->] (P14) -- (P35);
 \draw [->] (P24) -- (P35);
 \draw [->] (P44) -- (P35);
 \draw [->] (P44) -- (P55);
 \draw [->] (P35) -- (P15);
 \draw [->] (P35) -- (P25);
 \draw [->] (P35) -- (P45);
 \draw [->] (P55) -- (P45);
 \draw [->] (P15) -- (P36);
 \draw [->] (P25) -- (P36);
 \draw [->] (P45) -- (P36);
 \draw [->] (P45) -- (P56);
\end{tikzpicture} \]
Here $L_X = \oplus_{i=1}^4 L_i$ and $R_X = \oplus_{i=1}^4 R_i$. Let $\Sigma$ be the completion of $L_X$ to a local slice shown in light grey and $\Sigma'$ the completion of $R_X$ to a local slice shown in dark grey. Note that the completions are not unique. We have that $J_X^- \Sigma = \Sigma'$ and $J_X^+ \Sigma' = \Sigma$. The trench $X$ is a relative source with respect to  $\Sigma$ and a relative sink with respect to $\Sigma'$. Let $C = B / \Ann_B\Sigma$ and $C' = B / \Ann_B \Sigma'$. 
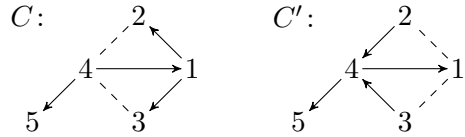
\[ \begin{tikzpicture}[scale=.7,yscale=-1]
 \node (B1) at (3,1) [inner sep=1pt] {1};
 \node (B2) at (2,0) [inner sep=1pt] {2};
 \node (B3) at (2,2) [inner sep=1pt] {3};
 \node (B4) at (1,1) [inner sep=1pt] {4};
 \node (B5) at (0,2) [inner sep=1pt] {5};
 \node (B) at (0,0) {$C \colon$};
 \node (C1) at (8,1) [inner sep=1pt] {1};
 \node (C2) at (7,0) [inner sep=1pt] {2};
 \node (C3) at (7,2) [inner sep=1pt] {3};
 \node (C4) at (6,1) [inner sep=1pt] {4};
 \node (C5) at (5,2) [inner sep=1pt] {5};
 \node (C) at (5,0) {$C'\colon$};
 \draw [->] (B4) -- (B5);
 \draw [->] (B4) -- (B1);
 \draw [->] (B1) -- (B2);
 \draw [->] (B1) -- (B3);
 \draw [dashed] (B2) -- (B4);
 \draw [dashed] (B3) -- (B4);
 \draw [->] (C4) -- (C5);
 \draw [->] (C4) -- (C1);
 \draw [dashed] (C1) -- (C2);
 \draw [dashed] (C1) -- (C3);
 \draw [->] (C2) -- (C4);
 \draw [->] (C3) -- (C4);
\end{tikzpicture} \]
Observe that these operations amount to exchanging the relations ending at the cell corresponding to the trench we jumped with arrows coming out of the cell, and the arrows coming in with relations. 
\end{example}

\section{Generalized 2-APR tilting} \label{section.2apr}

In this section we recall and generalize 2-APR tilting, which was originally introduced in \cite{IO}. We then show that ``jumping trenches'', as introduced in Section~\ref{section.Loc_slices}, is a special case of this generalized $2$-APR tilting. Finally we give an explicit description of the quiver and relations of the 2-APR tilted algebra in terms of the original algebra.

APR tilting has been introduced by Auslander, Platzeck and Reiten in \cite{APR}:

Assume $C$ is a basic algebra, and $C = C_0 \oplus C_R$ where $C_0$ is a simple projective $C$-module. Then $T = \tau^- C_0 \oplus C_R$ is a tilting module. If moreover the injective dimension $\id C_0 = 1$, then the quiver of $\End_C(T)^{\op}$ is obtained from the quiver of $C$ by reversing all arrows ending in the vertex corresponding to $C_0$.

The procedure of APR tilting was generalized in \cite{IO}. Here we are mostly interested in what is called 2-APR tilting in that paper. Instead of replacing $C_0$ by $\tau^- C_0$ it is replaced by the complex $\tau^- C_0[1] = F C_0$ (called $\tau_2^- C_0$ in that paper). Then, provided certain conditions are satisfied, the quiver with relations of the algebra $\End_C(\tau^- C_0[1] \oplus C_R)^{\op}$ can be read off directly from the quiver with relations of the algebra $C$ (see Proposition~\ref{proposition.apr-quiver} below).

Here we generalize that construction in two ways:

First, we use the replacement $F C_0$ (constructed in the derived category) instead of the construction of $\tau_2^-$ in the module category in \cite{IO}, and allow the result to be a proper complex.

Second, we do not require $C_0$ to be simple. In fact we will wish to apply the procedure to all indecomposable summands in one cell at once.

\begin{definition} \label{definition.n-APR}
Let $C$ be an algebra of global dimension two. Assume $C = C_0 \oplus C_R$ with $\Hom_{C}(C_R, C_0) = 0$ and $\Ext^1_{C}(\nu C_R, C_0) = 0$. Then we call $T := F C_0 \oplus C_R$ the 2-APR tilting complex associated to $C_0$, and $C' = \End_{D^b(\mod C)} (T)^{\op}$ the generalized 2-APR tilt of $C$ at $C_0$.
\end{definition}

For the application in this paper $C$ will be tilted, but it is not necessary to assume it to be tilted at this moment.

This definition is justified by the following fact:

\begin{lemma} \label{lemma.istilting}
In the situation of Definition~\ref{definition.n-APR} the complex $T$ is a tilting complex in $D^b(\mod C)$.
\end{lemma}

\begin{proof}
We start by showing
\begin{equation} \label{equation.nohoms}
\Hom_C(C_0, \nu C_R) = 0 = \Hom_C( \nu C_R, C_0).
\end{equation}
Since $\Hom_{C}(C_R, C_0) = 0$ there are no arrows in the quiver of $C$ from vertices corresponding to $C_0$ to vertices corresponding to $C_R$. Hence all composition factors of $C_0$ are in $\add (C_0 / \rad C_0)$, and all composition factors of $\nu C_R$ are in $\add (C_R / \rad C_R)$. In particular $C_0$ and $\nu C_R$ have no common composition factors. This implies that (\ref{equation.nohoms}) holds.

Next we show that $T$ generates the derived category. Let $X \in D^b(\mod C)$ such that $\Hom_{D^b(\mod C)}(T, X[i]) = 0$ for all $ i \in \mathbb{Z}$. Then in particular $\Hom_{D^b(\mod C)}(C_R, X[i]) = 0$ for all $ i \in \mathbb{Z}$, and hence all composition factors of all homologies of $X$ are in $\add (C_0 / \rad C_0)$. Therefore $X$ is isomorphic to a complex with terms in $\add C_0$. On the other hand
\[ \Hom_{D^b(\mod C)}(C_0, \nu X[i]) = \Hom_{D^b(\mod C)}(F C_0, X[i+2]) = 0 \, \forall i \in \mathbb{Z}. \]
Hence $\nu X$ is isomorphic to a complex with terms in $\add \nu C_R$. Now
\[ \Hom_{D^b(\mod C)}(X,X) = D\Hom_{D^b(\mod C)}(X,\nu X) = 0, \]
and hence $X = 0$. So $T$ generates $D^b(\mod C)$.

It remains to see that $\Hom_{D^b(\mod C)}(T, T[i]) = 0$ for all $ i \neq 0$. Since $C_0$ and $C_R$ are projective modules we have
\[ \Hom_{D^b(\mod C)}(C_0, C_0[i]) = 0 = \Hom_{D^b(\mod C)}(C_R, C_R[i]) = 0 \qquad \forall i \neq 0, \]
and since $F$ is an autoequivalence of $D^b(\mod C)$ also
\[ \Hom_{D^b(\mod C)}(F C_0, F C_0[i]) = 0 \qquad \forall i \neq 0. \]
Next we see that
\begin{align*}
\Hom_{D^b(\mod C)}(F C_0, C_R[i]) & = \Hom_{D^b(\mod C)}(C_0, \nu C_R[i-2]) \\
& = D \Hom_{D^b(\mod C)}(C_R[i-2], C_0) \\
& = 0 \qquad \forall i.
\end{align*}
Finally we have
\[ \Hom_{D^b(\mod C)}(C_R, F C_0[i]) = \Hom_{D^b(\mod C)}(\nu C_R, C_0[i+2]). \]
Since $\gldim C = 2$, this vanishes for all $i \not\in \{-2, -1, 0\}$. For $i = -1$ it vanishes by assumption, and for $i = -2$ we have $\Hom_{D^b(\mod C)}(\nu C_R, C_0) = 0$ by (\ref{equation.nohoms}).
\end{proof}

\begin{remark}
Note that
\[ \Hom_{D^b(\mod C)}(C_R, F C_0) = \Ext^2_{C}(\nu C_R, C_0). \]
\end{remark}

The following lemma shows that jumping trenches (as introduced in Construction~\ref{construction_jumping}), or more generally passing from one local slice to another, are special cases of 2-APR tilting.

\begin{lemma} \label{lemma.slicegoodforapr}
Let $C$ be an iterated tilted algebra with $\gldim C = 2$. We decompose $C = \oplus_{a} C_a$ with $C_a$ indecomposable. Let $\Sigma \subseteq D^b(\mod C)$ be a complete slice which does not contain any of the $C_a$. Then for
\[ C_0 = \bigoplus_{\substack{\{a \mid \exists \text{ path} \\ C_a \leadsto \Sigma\}}} C_a \qquad \text{ and } \qquad C_R = \bigoplus_{\substack{\{a \mid \exists \text{ path } \\ \Sigma \leadsto C_a\}}} C_a \]
the assumptions of Definition~\ref{definition.n-APR} are satisfied. That is, we have $\Hom_{C}(C_R, C_0) = 0 = \Ext^1_{C}(\nu C_R, C_0)$.
\end{lemma}

\begin{proof}
The first claim holds by construction, the second follows immediately from the fact that $\nu C_R = \tau C_R[1]$.
\end{proof}

Assume now that $C$ is tilted. Then $C$ is obtained from the corresponding cluster-tilted algebra $B = C \ltimes \Ext^2_{C}(DC, C)$ by factoring out the arrows in some admissible set $S$.

The next proposition explicitly gives us the quiver of any generalized 2-APR tilt of $C$.

\begin{proposition} \label{proposition.apr-quiver}
Let $C = B / (S)$ be tilted, where $B$ is the corresponding cluster-tilted algebra, and $S$ is an admissible set. Assume $C_0 \leq_{\oplus} C$ admits a 2-APR tilting complex. Then $\End_{D^b(\mod C)}(F C_0 \oplus C_R)^{\op}$ is isomorphic to $C' = B / (S')$, where
\[ S' = S \setminus \{\text{all arrows from $C_0$ to $C_R$}\} \cup \{\text{all arrows from $C_R$ to $C_0$}\}. \]
(Here ``all arrows'' refers to all arrows in the quiver of $B$.)
\end{proposition}

\begin{proof}
We denote the indecomposable projective modules over $C$ and $C'$ with simple top corresponding to vertex $a$ by $C_a$ and $C'_a$, respectively. Moreover we write
\[ \widetilde{C}_a = \left\{ \begin{matrix} F C_a & \text{ if } C_a \in \add C_0 \\ C_a & \text{ if } C_a \in \add C_R. \end{matrix} \right. \]
Then we have to show that $\Hom_{D^b(\mod C)} (\widetilde{C}_a, \widetilde{C}_b) = \Hom_{C'}( C'_a, C'_b)$ for any $a$ and $b$.

By construction (see the proof of \ref{lemma.istilting}) the morphisms inside $C_0$ and the morphisms inside $C_R$ are not affected by the tilt (and neither by our change of admissible set), so the claim holds if either both or none of $C_a$ and $C_b$ are in $\add C_0$. Moreover we have seen that $\Hom_{D^b(\mod C)}(F C_0, C_R) = 0$, and since all arrows from $C_R$ to $C_0$ are contained in $S'$ we have $\Hom_{C'}( C'_a, C'_b) = 0$ if $C_a \in \add C_0$ and $C_b \in \add C_R$. Finally for $C_a \in \add C_R$ and $C_b \in \add C_0$ we have
\begin{align*}
\Hom_{D^b(\mod C)} (\widetilde{C}_a, \widetilde{C}_b) & = \Ext^2( \nu C_a, C_b) \\
& = \Hom_{B}(C_a, C_b) \\
& = \Hom_{C'}(C'_a, C'_b)
\end{align*}
as claimed.
\end{proof}

\begin{remark}
Proposition~\ref{proposition.apr-quiver} holds more generally for any finite dimensional algebra $C$ of global dimension 2. In that case one uses $B = T_{C} \Ext^2(DC, C)$, the tensor algebra of $\Ext^2(DC, C)$ over $C$. This algebra is the endomorphism ring of the image of $B$ in the cluster category of $B$, as defined by Amiot (see \cite{CC, C_PhD}).
\end{remark}

\begin{example}
Let us now look at what the construction of Lemma~\ref{lemma.slicegoodforapr} and Proposition~\ref{proposition.apr-quiver} does in the setup of the Example~\ref{example.trenches}. Recall that $C$ was given by the quiver with relations
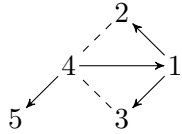
\[ \begin{tikzpicture}[scale=.7,yscale=-1]
 \node (P1) at (3,1) [inner sep=1pt] {1};
 \node (P2) at (2,0) [inner sep=1pt] {2};
 \node (P3) at (2,2) [inner sep=1pt] {3};
 \node (P4) at (1,1) [inner sep=1pt] {4};
 \node (P5) at (0,2) [inner sep=1pt] {5};
 \draw [->] (P4) -- (P5);
 \draw [->] (P4) -- (P1);
 \draw [->] (P1) -- (P2);
 \draw [->] (P1) -- (P3);
 \draw [dashed] (P2) -- (P4);
 \draw [dashed] (P3) -- (P4);
\end{tikzpicture} \]
and the Auslander-Reiten quiver of its derived category looks as follows (continuing infinitely in both directions):
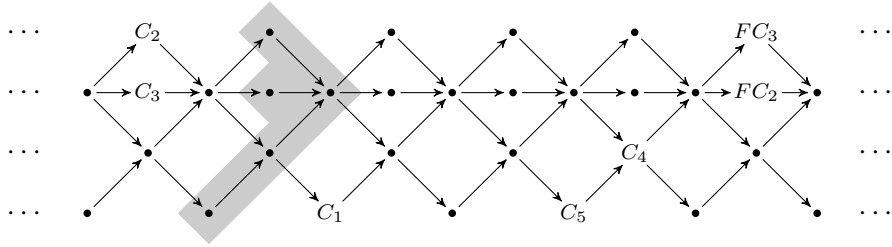
\[ \begin{tikzpicture}[scale=.8,yscale=-1]
 \fill [fill1] (3,.5) -- (4.5,2) -- (2,4.5) -- (1.5,4) -- (3,2.5) -- (2.5,2) -- (3,1.5) -- (2.5,1) -- cycle;
 \foreach \x in {0,...,6}
  \foreach \y in {2,4}
   \node (\y-\x-e) at (\x*2,\y) [vertex] {};
 \foreach \x in {0,...,5}
  \foreach \y in {1,2,3}
   \node (\y-\x-o) at (\x*2+1,\y) [vertex] {};
 \replacevertex{(1-0-o)}{[tvertex] {$C_2$}}
 \replacevertex{(2-0-o)}{[tvertex] {$C_3$}}
 \replacevertex{(4-2-e)}{[tvertex] {$C_1$}}
 \replacevertex{(4-4-e)}{[tvertex] {$C_5$}}
 \replacevertex{(3-4-o)}{[tvertex] {$C_4$}}
 \replacevertex{(1-5-o)}{[tvertex] {$FC_3$}}
 \replacevertex{(2-5-o)}{[tvertex] {$FC_2$}}
 \foreach \xa/\xb in {0/1,1/2,2/3,3/4,4/5,5/6}
  \foreach \ya/\yb in {2/1,2/2,2/3,4/3}
   {
    \draw [->] (\ya-\xa-e) -- (\yb-\xa-o);
    \draw [->] (\yb-\xa-o) -- (\ya-\xb-e);
   }
 \foreach \y in {1,...,4}
  {
   \node at (-1, \y) {$\cdots$};
   \node at (13, \y) {$\cdots$};
  }
\end{tikzpicture} \]
We choose a complete slice not containing any of the $C_a$ as indicated by the grey area above. Then, in the construction of Lemma~\ref{lemma.slicegoodforapr} we obtain $C_0 = C_2 \oplus C_3$ and $C_R = C_1 \oplus C_4 \oplus C_5$. Now the quiver with relations of $\End_{D^b(\mod C)}(F C_0 \oplus C_R)^{\op}$ is
\[ \begin{tikzpicture}[scale=.7,yscale=-1]
 \node (P1) at (3,1) [inner sep=1pt] {1};
 \node (P2) at (2,0) [inner sep=1pt] {2};
 \node (P3) at (2,2) [inner sep=1pt] {3};
 \node (P4) at (1,1) [inner sep=1pt] {4};
 \node (P5) at (0,2) [inner sep=1pt] {5};
 \draw [->] (P4) -- (P5);
 \draw [->] (P4) -- (P1);
 \draw [dashed] (P1) -- (P2);
 \draw [dashed] (P1) -- (P3);
 \draw [->] (P2) -- (P4);
 \draw [->] (P3) -- (P4);
\end{tikzpicture} \]
This follows from Proposition~\ref{proposition.apr-quiver}. It can also be verified by looking directly at the Auslander-Reiten quiver above.
\end{example}

\section{The algorithm} \label{section_algorithm}

In this section we put together the techniques developed in Sections~\ref{section.Loc_slices} and \ref{section.2apr} to obtain an algorithm that, given a cluster-tilted algebra of finite type $B$, produces all maximal tilted subalgebras.

For the rest of the section let $B$ be the input to our algorithm, that is some fixed cluster-tilted algebra of Dynkin type.

\begin{step} \label{step.location}
Determine the distribution of the indecomposable direct summands of a cluster-tilting object $T$ in a cluster category $\mathcal{C}$ with $\End_{\cC}(T)^{\op} = B$.
\end{step}

\begin{remark}
We refer the reader to \cite{BOW} for a technique to find the distribution of a cluster-tilting object in the AR-quiver of the cluster category.
\end{remark}

\begin{step}
Determine which indecomposable direct summands of $T$ lie in the same cell.
\end{step}

This can be done by directly applying the definition of the equivalence relation $\equiv$ (see Definition~\ref{definition.celldecomp}).

\begin{step}
Choose a local slice $\Sigma$ such that $\tau^- \Sigma \cap \add T = 0$.
\end{step}

\begin{step}
Determine a tilted admissible set $S$ such that $B / \Ann_B \Sigma = B / (S)$. Call this tilted algebra $C$.
\end{step}

We can read off the tilted admissible set $S$ from the AR-quiver of $\mathcal{C}$ as follows: $S$ consists of arrows $b \to a$ in the quiver $Q_B$ of $B$, such that $\tau I (T_a, T_b) \cap \Sigma \neq \emptyset$ (see Theorem~\ref{theorem.generate_annihilator}~(b)).

\begin{step} \label{step.r}
Move $\Sigma$ as far to the right as possible within its homotopy class.
\end{step}

By Theorem~\ref{theorem.homotopy} this step does not change the tilted algebra $C$, and hence neither the tilted admissible set $S$.

\begin{step} \label{step.choosejump}
For any cell $\widetilde{T}$ which is a relative source with respect to $\Sigma$ and such that $\widetilde{T} \in \tau^{-2} \Sigma$, jump the trench $\tau \widetilde{T}$ as in Construction~\ref{construction_jumping}. We call the local slice obtained in this way $\Sigma_{\widetilde{T}}$.
\end{step}

By Proposition~\ref{proposition.apr-quiver} this amounts to the following:
\begin{itemize}
\item Removing all arrows $i \to j$, where $T_i$ is in the cell $\widetilde{T}$ and $T_j$ in some other cell, from the set $S$.
\item Adding all arrows $i \to j$ in $Q_B$, where $T_j$ is in the cell $\widetilde{T}$ and $T_i$ in some other cell, to the set $S$.
\end{itemize}
Let $S_{\wT}$ be the new tilted admissible set obtained in this way. Then $C_{\wT} = B / (S_{\wT}) = B / \Ann_B \Sigma_{\wT}$.

\begin{step}
Apply the algorithm starting in Step~\ref{step.r} to the new tilted admissible sets and tilted algebras until no new maximal tilted subalgebras are obtained any more.
\end{step}

\begin{remark}
We could also apply the procedure in the opposite direction (that is, move the local slice to the left).
\end{remark}

\begin{example}
Let $B$ be the cluster-tilted algebra with the following quiver.
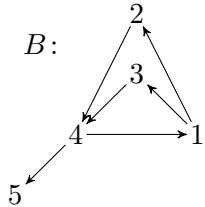
\[ \begin{tikzpicture}[scale=.8,yscale=-1]
 \node (5) at (11,3) [inner sep=1pt] {5};
 \node (4) at (12,2) [inner sep=1pt] {4};
 \node (3) at (13,1) [inner sep=1pt] {3};
 \node (2) at (13,0) [inner sep=1pt] {2};
 \node (1) at (14,2) [inner sep=1pt] {1};
 \node at (11.5,.5) {$B \colon$};
 \draw [->] (1) -- (2);
 \draw [->] (1) -- (3);
 \draw [->] (3) -- (4);
 \draw [->] (2) -- (4);
 \draw [->] (4) -- (5);
 \draw [->] (4) -- (1);
\end{tikzpicture} \]
\setcounter{step}{0}
\begin{step}
Observe that the summands of the cluster-tilting object $T$ having endomorphism ring $B$ are distributed in the cluster category of $D_5$ as follows:
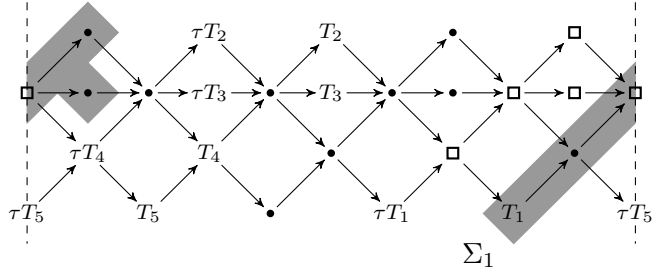
\[ \begin{tikzpicture}[scale=.8,yscale=-1]
 \fill [fill2] (1,-0.5) -- (0,0.5) -- (0,1.5) -- (0.5,1)   -- (1,1.5) -- (1.5,1) -- (1,0.5) -- (1.5,0) -- cycle; 
 \fill [fill2] (10,1.5) -- (10,0.5) -- (7.5,3) --  (8,3.5) --  cycle;
 \node (P31) at (0,1) [mvertex] {};
 \node (P51) at (0,3) [tvertex] {$\tau T_5$};
 \node (P11) at (1,0) [vertex] {};
 \node (P21) at (1,1) [vertex] {};
 \node (P41) at (1,2) [tvertex] {$\tau T_4$};
 \node (P32) at (2,1)  [vertex] {};
 \node (P52) at (2,3)  [tvertex] {$T_5$};
 \node (P12) at (3,0) [tvertex] {$\tau T_2$};
 \node (P22) at (3,1) [tvertex] {$\tau T_3$};
 \node (P42) at (3,2) [tvertex] {$T_4$};
 \node (P33) at (4,1) [vertex] {};
 \node (P53) at (4,3)  [vertex] {};
 \node (P13) at (5,0) [tvertex] {$T_2$};
 \node (P23) at (5,1) [tvertex] {$T_3$};
 \node (P43) at (5,2) [vertex] {};
 \node (P34) at (6,1) [vertex] {};
 \node (P54) at (6,3) [tvertex] {$\tau T_1$};
 \node (P14) at (7,0) [vertex] {};
 \node (P24) at (7,1) [vertex] {};
 \node (P44) at (7,2) [mvertex] {};
 \node (P35) at (8,1) [mvertex] {};
 \node (P55) at (8,3) [tvertex] {$T_1$};
 \node (P15) at (9,0) [mvertex] {};
 \node (P25) at (9,1) [mvertex] {};
 \node (P45) at (9,2) [vertex] {};
 \node (P36) at (10,1) [mvertex] {};
 \node (P56) at (10,3) [tvertex] {$\tau T_5$};
 \node [below left=1mm] at (P55.south) {$\Sigma_1$};
 \draw [dashed] (0,-.5) -- (P51);
 \draw [dashed] (P51) -- (0,3.5);
 \draw [dashed] (10,-.5) -- (P56);
 \draw [dashed] (P56) -- (10,3.5);
 \draw [->] (P31) -- (P11);
 \draw [->] (P31) -- (P21);
 \draw [->] (P31) -- (P41);
 \draw [->] (P51) -- (P41);
 \draw [->] (P11) -- (P32);
 \draw [->] (P21) -- (P32);
 \draw [->] (P41) -- (P32);
 \draw [->] (P41) -- (P52);
 \draw [->] (P32) -- (P12);
 \draw [->] (P32) -- (P22);
 \draw [->] (P32) -- (P42);
 \draw [->] (P52) -- (P42);
 \draw [->] (P12) -- (P33);
 \draw [->] (P22) -- (P33);
 \draw [->] (P42) -- (P33);
 \draw [->] (P42) -- (P53);
 \draw [->] (P33) -- (P13);
 \draw [->] (P33) -- (P23);
 \draw [->] (P33) -- (P43);
 \draw [->] (P53) -- (P43);
 \draw [->] (P13) -- (P34);
 \draw [->] (P23) -- (P34);
 \draw [->] (P43) -- (P34);
 \draw [->] (P43) -- (P54);
 \draw [->] (P34) -- (P14);
 \draw [->] (P34) -- (P24);
 \draw [->] (P34) -- (P44);
 \draw [->] (P54) -- (P44);
 \draw [->] (P14) -- (P35);
 \draw [->] (P24) -- (P35);
 \draw [->] (P44) -- (P35);
 \draw [->] (P44) -- (P55);
 \draw [->] (P35) -- (P15);
 \draw [->] (P35) -- (P25);
 \draw [->] (P35) -- (P45);
 \draw [->] (P55) -- (P45);
 \draw [->] (P15) -- (P36);
 \draw [->] (P25) -- (P36);
 \draw [->] (P45) -- (P36);
 \draw [->] (P45) -- (P56);
\end{tikzpicture} \]
\end{step}
\begin{step}
We see from the diagram above that the cells are $T_1$, $T_2 \oplus T_3$, and $T_4 \oplus T_5$, and hence the trenches are $\tau T_1$, $\tau T_2 \oplus \tau T_3$, and $\tau T_4 \oplus \tau T_5$.
\end{step}
\begin{step}
We choose our first local slice $\Sigma_1$ as indicated in the figure above.
\end{step}
\begin{step}
Since the only set which has non-empty intersection with $\Sigma_1$ is $\tau I(T_1,T_4)$ (this is the set indicated by the squares in the figure above), the corresponding tilted admissible set is $\{4\to 1\}$ and thus we obtain the maximal tilted subalgebra $C_1$, illustrated in the figure below.
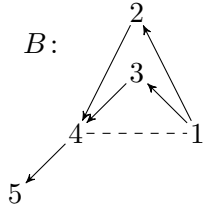
\[ \begin{tikzpicture}[scale=.8,yscale=-1]
 \node (5) at (11,3) [inner sep=1pt] {5};
 \node (4) at (12,2) [inner sep=1pt] {4};
 \node (3) at (13,1) [inner sep=1pt] {3};
 \node (2) at (13,0) [inner sep=1pt] {2};
 \node (1) at (14,2) [inner sep=1pt] {1};
 \node at (11.5,.5) {$B \colon$};
 \draw [->] (1) -- (2);
 \draw [->] (1) -- (3);
 \draw [->] (3) -- (4);
 \draw [->] (2) -- (4);
 \draw [->] (4) -- (5);
 \draw [dashed] (4) -- (1);
\end{tikzpicture} \]
\end{step}
\begin{step}
The local slice $\Sigma_1$ is already as far to the right as possible.
\end{step}
\begin{step}
We note that the only relative source with respect to $\Sigma_1$ is $T_4 \oplus T_5$. Jumping the corresponding trench we obtain the new tilted admissible set $\{ 2 \to 4, 3 \to 4 \}$.
\end{step}
\begin{step}
See Figure~\ref{figure.explaining_example} for all maximal tilted algebras obtained by repeatedly applying the last three steps.
\end{step}
\end{example}

\begin{figure}[htb]
\[ \begin{tikzpicture}[baseline,scale = 0.8,yscale=-1]
 \fill [fill1] (4.5,0) -- (5,0.5) -- (4.5,1) -- (5,1.5) -- (5.5,1) -- (8,3.5) -- (10,1.5) -- (10,0.5) -- (9,-0.5) -- (5,-0.5) -- cycle;
 \fill [fill2] (1,-0.5) -- (0,0.5) -- (0,1.5) -- (0.5,1)   -- (1,1.5) -- (1.5,1) -- (1,0.5) -- (1.5,0) -- cycle; 
 \fill [fill2] (10,1.5) -- (10,0.5) -- (7.5,3) --  (8,3.5) --  cycle;
 \node (P31) at (0,1) [vertex] {};
 \node (P51) at (0,3) [tvertex] {$\tau T_5$};
 \node (P11) at (1,0) [vertex] {};
 \node (P21) at (1,1) [vertex] {};
 \node (P41) at (1,2) [tvertex] {$\tau T_4$};
 \node (P32) at (2,1)  [vertex] {};
 \node (P52) at (2,3)  [vertex] {};
 \node (P12) at (3,0) [tvertex] {$\tau T_2$};
 \node (P22) at (3,1) [tvertex] {$\tau T_3$};
 \node (P42) at (3,2) [vertex] {};
 \node (P33) at (4,1) [vertex] {};
 \node (P53) at (4,3)  [vertex] {};
 \node (P13) at (5,0) [vertex] {};
 \node (P23) at (5,1) [vertex] {};
 \node (P43) at (5,2) [vertex] {};
 \node (P34) at (6,1) [vertex] {};
 \node (P54) at (6,3) [tvertex] {$\tau T_1$};
 \node (P14) at (7,0) [vertex] {};
 \node (P24) at (7,1) [vertex] {};
 \node (P44) at (7,2) [vertex] {};
 \node (P35) at (8,1) [vertex] {};
 \node (P55) at (8,3) [vertex] {};
 \node (P15) at (9,0) [vertex] {};
 \node (P25) at (9,1) [vertex] {};
 \node (P45) at (9,2) [vertex] {};
 \node (P36) at (10,1) [vertex] {};
 \node (P56) at (10,3) [tvertex] {$\tau T_5$};
 \node [below left=1mm] at (P55.south) {$\Sigma_1$};
 \node at (-.5,0) {a)};
 \draw [dashed] (0,-.5) -- (P51);
 \draw [dashed] (P51) -- (0,3.5);
 \draw [dashed] (10,-.5) -- (P56);
 \draw [dashed] (P56) -- (10,3.5);
 \draw [->] (P31) -- (P11);
 \draw [->] (P31) -- (P21);
 \draw [->] (P31) -- (P41);
 \draw [->] (P51) -- (P41);
 \draw [->] (P11) -- (P32);
 \draw [->] (P21) -- (P32);
 \draw [->] (P41) -- (P32);
 \draw [->] (P41) -- (P52);
 \draw [->] (P32) -- (P12);
 \draw [->] (P32) -- (P22);
 \draw [->] (P32) -- (P42);
 \draw [->] (P52) -- (P42);
 \draw [->] (P12) -- (P33);
 \draw [->] (P22) -- (P33);
 \draw [->] (P42) -- (P33);
 \draw [->] (P42) -- (P53);
 \draw [->] (P33) -- (P13);
 \draw [->] (P33) -- (P23);
 \draw [->] (P33) -- (P43);
 \draw [->] (P53) -- (P43);
 \draw [->] (P13) -- (P34);
 \draw [->] (P23) -- (P34);
 \draw [->] (P43) -- (P34);
 \draw [->] (P43) -- (P54);
 \draw [->] (P34) -- (P14);
 \draw [->] (P34) -- (P24);
 \draw [->] (P34) -- (P44);
 \draw [->] (P54) -- (P44);
 \draw [->] (P14) -- (P35);
 \draw [->] (P24) -- (P35);
 \draw [->] (P44) -- (P35);
 \draw [->] (P44) -- (P55);
 \draw [->] (P35) -- (P15);
 \draw [->] (P35) -- (P25);
 \draw [->] (P35) -- (P45);
 \draw [->] (P55) -- (P45);
 \draw [->] (P15) -- (P36);
 \draw [->] (P25) -- (P36);
 \draw [->] (P45) -- (P36);
 \draw [->] (P45) -- (P56);
\end{tikzpicture}
\quad
\begin{tikzpicture}[baseline,scale=.8,yscale=-1]
 \node at (.5,.5) {$C_1 \colon$};
 \node (5) at (0,3) [inner sep=1pt] {5};
 \node (4) at (1,2) [inner sep=1pt] {4};
 \node (3) at (2,1) [inner sep=1pt] {3};
 \node (2) at (2,0) [inner sep=1pt] {2};
 \node (1) at (3,2) [inner sep=1pt] {1};
 \draw [->] (1) -- (2);
 \draw [->] (1) -- (3);
 \draw [->] (3) -- (4);
 \draw [->] (2) -- (4);
 \draw [->] (4) -- (5);
 \draw [dashed] (4) -- (1);
\end{tikzpicture} \]
\[ \begin{tikzpicture}[baseline,scale = 0.8,yscale=-1]
 \fill [fill1] (0.5,0) -- (1,0.5) -- (0.5,1) -- (1,1.5) -- (1.5,1) -- (2.5,2) -- (1.5,3) -- (2,3.5) -- (4,3.5) -- (4.5,3) -- (1,-0.5) -- cycle;
 \fill [fill2] (0.5,0) -- (1,0.5) -- (0.5,1) -- (1,1.5) -- (1.5,1) -- (2.5,2) -- (1.5,3) -- (2,3.5) -- (2.5,3) -- (3.5,2) -- (1,-0.5) -- cycle;
 \node (P31) at (0,1) [vertex] {};
 \node (P51) at (0,3) [tvertex] {$\tau T_5$};
 \node (P11) at (1,0) [vertex] {};
 \node (P21) at (1,1) [vertex] {};
 \node (P41) at (1,2) [tvertex] {$\tau T_4$};
 \node (P32) at (2,1) [vertex] {};
 \node (P52) at (2,3) [vertex] {};
 \node (P12) at (3,0) [tvertex] {$\tau T_2$};
 \node (P22) at (3,1) [tvertex] {$\tau T_3$};
 \node (P42) at (3,2) [vertex] {};
 \node (P33) at (4,1) [vertex] {};
 \node (P53) at (4,3) [vertex] {};
 \node (P13) at (5,0) [vertex] {};
 \node (P23) at (5,1) [vertex] {};
 \node (P43) at (5,2) [vertex] {};
 \node (P34) at (6,1) [vertex] {};
 \node (P54) at (6,3) [tvertex] {$\tau T_1$};
 \node (P14) at (7,0) [vertex] {};
 \node (P24) at (7,1) [vertex] {};
 \node (P44) at (7,2) [vertex] {};
 \node (P35) at (8,1) [vertex] {};
 \node (P55) at (8,3) [vertex] {};
 \node (P15) at (9,0) [vertex] {};
 \node (P25) at (9,1) [vertex] {};
 \node (P45) at (9,2) [vertex] {};
 \node (P36) at (10,1) [vertex] {};
 \node (P56) at (10,3) [tvertex] {$\tau T_5$};
 \node [below left=1mm] at (P52.south) {$\Sigma_2$};
 \node at (-.5,0) {b)};
 \draw [dashed] (0,-.5) -- (P51);
 \draw [dashed] (P51) -- (0,3.5);
 \draw [dashed] (10,-.5) -- (P56);
 \draw [dashed] (P56) -- (10,3.5);
 \draw [->] (P31) -- (P11);
 \draw [->] (P31) -- (P21);
 \draw [->] (P31) -- (P41);
 \draw [->] (P51) -- (P41);
 \draw [->] (P11) -- (P32);
 \draw [->] (P21) -- (P32);
 \draw [->] (P41) -- (P32);
 \draw [->] (P41) -- (P52);
 \draw [->] (P32) -- (P12);
 \draw [->] (P32) -- (P22);
 \draw [->] (P32) -- (P42);
 \draw [->] (P52) -- (P42);
 \draw [->] (P12) -- (P33);
 \draw [->] (P22) -- (P33);
 \draw [->] (P42) -- (P33);
 \draw [->] (P42) -- (P53);
 \draw [->] (P33) -- (P13);
 \draw [->] (P33) -- (P23);
 \draw [->] (P33) -- (P43);
 \draw [->] (P53) -- (P43);
 \draw [->] (P13) -- (P34);
 \draw [->] (P23) -- (P34);
 \draw [->] (P43) -- (P34);
 \draw [->] (P43) -- (P54);
 \draw [->] (P34) -- (P14);
 \draw [->] (P34) -- (P24);
 \draw [->] (P34) -- (P44);
 \draw [->] (P54) -- (P44);
 \draw [->] (P14) -- (P35);
 \draw [->] (P24) -- (P35);
 \draw [->] (P44) -- (P35);
 \draw [->] (P44) -- (P55);
 \draw [->] (P35) -- (P15);
 \draw [->] (P35) -- (P25);
 \draw [->] (P35) -- (P45);
 \draw [->] (P55) -- (P45);
 \draw [->] (P15) -- (P36);
 \draw [->] (P25) -- (P36);
 \draw [->] (P45) -- (P36);
 \draw [->] (P45) -- (P56);
\end{tikzpicture}
\quad
\begin{tikzpicture}[baseline,scale=.8,yscale=-1]
 \node at (.5,.5) {$C_2 \colon$};
 \node (5) at (0,3) [inner sep=1pt] {5};
 \node (4) at (1,2) [inner sep=1pt] {4};
 \node (3) at (2,1) [inner sep=1pt] {3};
 \node (2) at (2,0) [inner sep=1pt] {2};
 \node (1) at (3,2) [inner sep=1pt] {1};
 \draw [->] (1) -- (2);
 \draw [->] (1) -- (3);
 \draw [dashed] (3) -- (4);
 \draw [dashed] (2) -- (4);
 \draw [->] (4) -- (5);
 \draw [->] (4) -- (1);
\end{tikzpicture} \]
\[ \begin{tikzpicture}[baseline,scale = 0.8,yscale=-1]
 \fill [fill1] (4.5,0) -- (1.5,3) -- (2,3.5) -- (4,3.5) -- (6.5,1) -- (7,1.5) -- (7.5,1) -- (7,0.5) -- (7.5,0) --  (7,-0.5) -- (5, -0.5) -- cycle;
 \fill [fill2] (4.5,0) -- (1.5,3) -- (2,3.5) -- (4.5,1) -- (5,1.5)-- (5.5,1) -- (5,0.5) -- (5.5,0) -- (5,-0.5) -- cycle;
 \fill [fill12] (6.5,0) -- (3.5,3) -- (4,3.5) -- (6.5,1) -- (7,1.5)-- (7.5,1) -- (7,0.5) -- (7.5,0) -- (7,-0.5) -- cycle;
 \node (P31) at (0,1) [vertex] {};
 \node (P51) at (0,3) [tvertex] {$\tau T_5$};
 \node (P11) at (1,0) [vertex] {};
 \node (P21) at (1,1) [vertex] {};
 \node (P41) at (1,2) [tvertex] {$\tau T_4$};
 \node (P32) at (2,1) [vertex] {};
 \node (P52) at (2,3) [vertex] {};
 \node (P12) at (3,0) [tvertex] {$\tau T_2$};
 \node (P22) at (3,1) [tvertex] {$\tau T_3$};
 \node (P42) at (3,2) [vertex] {};
 \node (P33) at (4,1) [vertex] {};
 \node (P53) at (4,3) [vertex] {};
 \node (P13) at (5,0) [vertex] {};
 \node (P23) at (5,1) [vertex] {};
 \node (P43) at (5,2) [vertex] {};
 \node (P34) at (6,1) [vertex] {};
 \node (P54) at (6,3) [tvertex] {$\tau T_1$};
 \node (P14) at (7,0) [vertex] {};
 \node (P24) at (7,1) [vertex] {};
 \node (P44) at (7,2) [vertex] {};
 \node (P35) at (8,1) [vertex] {};
 \node (P55) at (8,3) [vertex] {};
 \node (P15) at (9,0) [vertex] {};
 \node (P25) at (9,1) [vertex] {};
 \node (P45) at (9,2) [vertex] {};
 \node (P36) at (10,1) [vertex] {};
 \node (P56) at (10,3) [tvertex] {$\tau T_5$};
 \node [below left=1mm] at (P52.south) {$\Sigma_3$};
 \node [below left=1mm] at (P53.south) {$\Sigma'_3$};
 \node at (-.5,0) {c)};
 \draw [dashed] (0,-.5) -- (P51);
 \draw [dashed] (P51) -- (0,3.5);
 \draw [dashed] (10,-.5) -- (P56);
 \draw [dashed] (P56) -- (10,3.5);
 \draw [->] (P31) -- (P11);
 \draw [->] (P31) -- (P21);
 \draw [->] (P31) -- (P41);
 \draw [->] (P51) -- (P41);
 \draw [->] (P11) -- (P32);
 \draw [->] (P21) -- (P32);
 \draw [->] (P41) -- (P32);
 \draw [->] (P41) -- (P52);
 \draw [->] (P32) -- (P12);
 \draw [->] (P32) -- (P22);
 \draw [->] (P32) -- (P42);
 \draw [->] (P52) -- (P42);
 \draw [->] (P12) -- (P33);
 \draw [->] (P22) -- (P33);
 \draw [->] (P42) -- (P33);
 \draw [->] (P42) -- (P53);
 \draw [->] (P33) -- (P13);
 \draw [->] (P33) -- (P23);
 \draw [->] (P33) -- (P43);
 \draw [->] (P53) -- (P43);
 \draw [->] (P13) -- (P34);
 \draw [->] (P23) -- (P34);
 \draw [->] (P43) -- (P34);
 \draw [->] (P43) -- (P54);
 \draw [->] (P34) -- (P14);
 \draw [->] (P34) -- (P24);
 \draw [->] (P34) -- (P44);
 \draw [->] (P54) -- (P44);
 \draw [->] (P14) -- (P35);
 \draw [->] (P24) -- (P35);
 \draw [->] (P44) -- (P35);
 \draw [->] (P44) -- (P55);
 \draw [->] (P35) -- (P15);
 \draw [->] (P35) -- (P25);
 \draw [->] (P35) -- (P45);
 \draw [->] (P55) -- (P45);
 \draw [->] (P15) -- (P36);
 \draw [->] (P25) -- (P36);
 \draw [->] (P45) -- (P36);
 \draw [->] (P45) -- (P56);
\end{tikzpicture}
\quad
\begin{tikzpicture}[baseline,scale=.8,yscale=-1]
 \node at (.5,.5) {$C_3 \colon$};
 \node (5) at (0,3) [inner sep=1pt] {5};
 \node (4) at (1,2) [inner sep=1pt] {4};
 \node (3) at (2,1) [inner sep=1pt] {3};
 \node (2) at (2,0) [inner sep=1pt] {2};
 \node (1) at (3,2) [inner sep=1pt] {1};
 \draw [dashed] (1) -- (2);
 \draw [dashed] (1) -- (3);
 \draw [->] (3) -- (4);
 \draw [->] (2) -- (4);
 \draw [->] (4) -- (5);
 \draw [->] (4) -- (1);
\end{tikzpicture} \]
\[ \begin{tikzpicture}[baseline,scale = 0.8,yscale=-1]
  \fill [fill1] (4.5,0) -- (5,0.5) -- (4.5,1) -- (5,1.5) -- (5.5,1) -- (8,3.5) -- (10,1.5) -- (10,0.5) -- (9,-0.5) -- (5,-0.5) -- cycle;
 \fill [fill1] (1,-0.5) -- (0,0.5) -- (0,1.5) -- (0.5,1)   -- (1,1.5) -- (1.5,1) -- (1,0.5) -- (1.5,0) -- cycle; 
 \fill [fill2] (6.5,0) -- (5.5,1) -- (8,3.5) -- (8.5,3) -- (6.5,1) -- (7,1.5) -- (7.5,1) -- (7,0.5) -- (7.5,0) -- (7,-0.5) -- cycle;
 \node (P31) at (0,1) [vertex] {};
 \node (P51) at (0,3) [tvertex] {$\tau T_5$};
 \node (P11) at (1,0) [vertex] {};
 \node (P21) at (1,1) [vertex] {};
 \node (P41) at (1,2) [tvertex] {$\tau T_4$};
 \node (P32) at (2,1)  [vertex] {};
 \node (P52) at (2,3)  [vertex] {};
 \node (P12) at (3,0) [tvertex] {$\tau T_2$};
 \node (P22) at (3,1) [tvertex] {$\tau T_3$};
 \node (P42) at (3,2) [vertex] {};
 \node (P33) at (4,1) [vertex] {};
 \node (P53) at (4,3)  [vertex] {};
 \node (P13) at (5,0) [vertex] {};
 \node (P23) at (5,1) [vertex] {};
 \node (P43) at (5,2) [vertex] {};
 \node (P34) at (6,1) [vertex] {};
 \node (P54) at (6,3) [tvertex] {$\tau T_1$};
 \node (P14) at (7,0) [vertex] {};
 \node (P24) at (7,1) [vertex] {};
 \node (P44) at (7,2) [vertex] {};
 \node (P35) at (8,1) [vertex] {};
 \node (P55) at (8,3) [vertex] {};
 \node (P15) at (9,0) [vertex] {};
 \node (P25) at (9,1) [vertex] {};
 \node (P45) at (9,2) [vertex] {};
 \node (P36) at (10,1) [vertex] {};
 \node (P56) at (10,3) [tvertex] {$\tau T_5$};
 \node [below right=1mm] at (P55.south) {$\Sigma_4$};
 \node at (-.5,0) {d)};  
 \draw [dashed] (0,-.5) -- (P51);
 \draw [dashed] (P51) -- (0,3.5);
 \draw [dashed] (10,-.5) -- (P56);
 \draw [dashed] (P56) -- (10,3.5);
 \draw [->] (P31) -- (P11);
 \draw [->] (P31) -- (P21);
 \draw [->] (P31) -- (P41);
 \draw [->] (P51) -- (P41);
 \draw [->] (P11) -- (P32);
 \draw [->] (P21) -- (P32);
 \draw [->] (P41) -- (P32);
 \draw [->] (P41) -- (P52);
 \draw [->] (P32) -- (P12);
 \draw [->] (P32) -- (P22);
 \draw [->] (P32) -- (P42);
 \draw [->] (P52) -- (P42);
 \draw [->] (P12) -- (P33);
 \draw [->] (P22) -- (P33);
 \draw [->] (P42) -- (P33);
 \draw [->] (P42) -- (P53);
 \draw [->] (P33) -- (P13);
 \draw [->] (P33) -- (P23);
 \draw [->] (P33) -- (P43);
 \draw [->] (P53) -- (P43);
 \draw [->] (P13) -- (P34);
 \draw [->] (P23) -- (P34);
 \draw [->] (P43) -- (P34);
 \draw [->] (P43) -- (P54);
 \draw [->] (P34) -- (P14);
 \draw [->] (P34) -- (P24);
 \draw [->] (P34) -- (P44);
 \draw [->] (P54) -- (P44);
 \draw [->] (P14) -- (P35);
 \draw [->] (P24) -- (P35);
 \draw [->] (P44) -- (P35);
 \draw [->] (P44) -- (P55);
 \draw [->] (P35) -- (P15);
 \draw [->] (P35) -- (P25);
 \draw [->] (P35) -- (P45);
 \draw [->] (P55) -- (P45);
 \draw [->] (P15) -- (P36);
 \draw [->] (P25) -- (P36);
 \draw [->] (P45) -- (P36);
 \draw [->] (P45) -- (P56);
\end{tikzpicture}
\quad
\begin{tikzpicture}[baseline,scale=.8,yscale=-1]
 \node at (.5,.5) {$C_4 \colon$};
 \node (5) at (0,3) [inner sep=1pt] {5};
 \node (4) at (1,2) [inner sep=1pt] {4};
 \node (3) at (2,1) [inner sep=1pt] {3};
 \node (2) at (2,0) [inner sep=1pt] {2};
 \node (1) at (3,2) [inner sep=1pt] {1};
 \draw [->] (1) -- (2);
 \draw [->] (1) -- (3);
 \draw [->] (3) -- (4);
 \draw [->] (2) -- (4);
 \draw [->] (4) -- (5);
 \draw [dashed] (4) -- (1);
\end{tikzpicture} \] 
\caption{\label{figure.explaining_example} In this figure, the light grey areas denote the equivalence class of the local slices shown in darker grey. To the right we show the corresponding maximal tilted subalgebras. 
Observe that $\Sigma_2 = J^-_{\tau T_4 \oplus \tau T_5} \Sigma_1$, $\Sigma_3 = J^-_{\tau T_2 \oplus \tau T_3} \Sigma_2$ is homotopic to $\Sigma'_3$ and $\Sigma_4 = J^-_{\tau T_1} \Sigma'_3$ is homotopic to $\Sigma_1$. Furthermore, $[\Sigma_1],[\Sigma_2]$ and $[\Sigma_3]$ are all the equivalence classes of local slices in $\cC \setminus \add(\tau T)$ (or equivalently in $\mod B$).}
\end{figure}
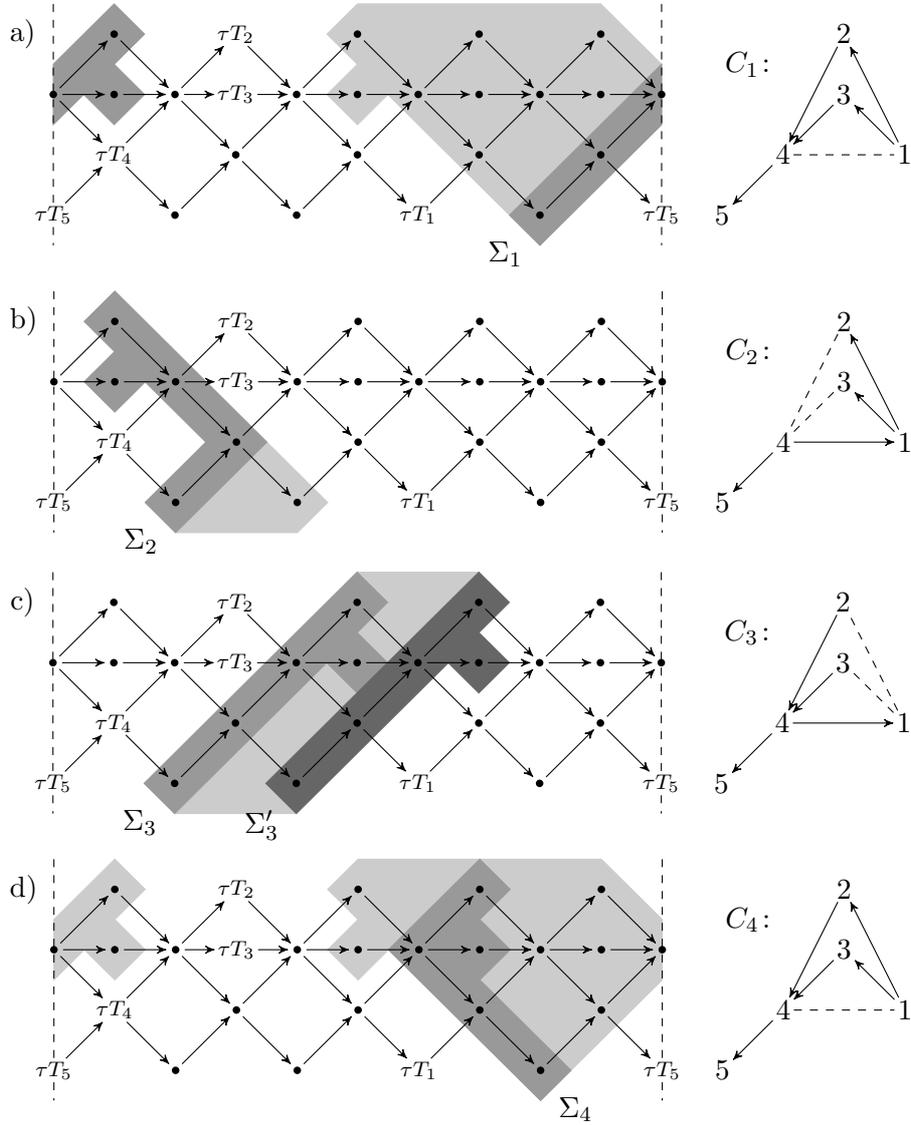

The following example shows that in Step~\ref{step.choosejump} we have to follow the local slice. Some relative sources cannot be jumped.

\begin{example} Let $C_1$ be the tilted algebra of type $A_5$ shown below. The cell corresponding to vertex $3$ is a relative source. If we apply the 2-APR tilt at the indecomposable projective $C_{1}$-module at vertex $3$ we obtain the algebra $C_2$ which is iterated tilted of type $A_5$ but not tilted. 
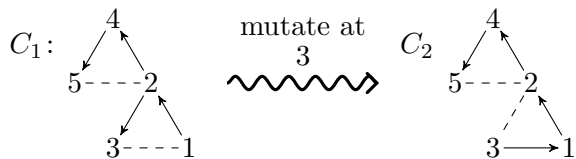
\[ \begin{tikzpicture}[yscale=-1]
\node (51) at (1,1) [inner sep=1pt] {$5$};
\node (41) at (1.5,0.14) [inner sep=1pt] {$4$};
\node (31) at (1.5,1.86) [inner sep=1pt] {$3$};
\node (21) at (2,1) [inner sep=1pt] {$2$};
\node (11) at (2.5,1.86) [inner sep=1pt] {$1$};
\node (C1) at (.5,.5) {$C_1 \colon$};
\draw [->] (11) -- (21);
\draw [->] (21) -- (31);
\draw [dashed] (31) -- (11);
\draw [->] (21) -- (41);
\draw [->] (41) -- (51);
\draw [dashed] (51) -- (21);
 \draw [leadsto] (3,1) -- (5,1);
  \node [above] at (4,0.5) {mutate at};   
  \node [above] at (4,0.9) {$3$};
\node (52) at (6,1) [inner sep=1pt] {$5$};
\node (42) at (6.5,.14) [inner sep=1pt] {$4$};
\node (32) at (6.5,1.86) [inner sep=1pt] {$3$};
\node (22) at (7,1) [inner sep=1pt] {$2$};
\node (12) at (7.5,1.86) [inner sep=1pt] {$1$};
\node (C2) at (5.5,.5) {$C_2$};
\draw [->] (12) -- (22);
\draw [dashed] (22) -- (32);
\draw [->] (32) -- (12);
\draw [->] (22) -- (42);
\draw [->] (42) -- (52);
\draw [dashed] (52) -- (22);
\end{tikzpicture}\]
\end{example}

\section{Representation infinite cluster-tilted algebras} \label{section_infinite}

In this section, we explain how the theory developed in Sections~\ref{section.Loc_slices} to \ref{section_algorithm} can be generalized to find all the maximal tilted subalgebras of a cluster-tilted algebra $B$ of infinite type. We assume that we know the distribution of the direct summands of the cluster-tilting object in the AR-quiver of the cluster category.

The main task is to generalize the results of Section~\ref{section.Loc_slices} to this more general setup. 

First, observe that Theorem~\ref{theorem.generate_annihilator} holds for an arbitrary cluster tilted algebra. In this case, we might have multiple arrows between a pair of vertices. Let $\alpha \from b \to a$ be an arrow in an admissible tilted set $S$, and assume that there is another arrow $\beta$ from $b$ to $a$. We claim that $\beta$ belongs to $S$. To see this, recall that by \cite[2.4]{Hubner}, only one of the spaces $\Ext_C^i(S_a,S_b)$ can be non-zero for $i=0,1,2$, where $S_a, S_b$ are the simple $C$-modules at the vertices $a$ and $b$ for the tilted algebra $C=B/\langle S \rangle$. Note that the arrow $\alpha$ in $S$ corresponds to a minimal relation in $\Ext^2_C(S_a,S_b)\neq 0$. Therefore $\beta$ also corresponds to a minimal relation in the same space, and thus $\beta$ belongs to $S$.

 Second, notice that Proposition~\ref{prop.tilted_admissible.loc_slice} relies only on Theorem~\ref{theorem.generate_annihilator}~(b), and thus holds in this generality.

We will generalize Definitions~\ref{definition.homotopy} and \ref{definition.celldecomp}. This is done for two reasons: First, to deal with the fact that, in general, there is a finite number of indecomposable objects lying in the connecting component of $\mod B$ that does not belong to any local slice (see \cite[22]{ABS2}). Second, to deal with the possible regular summands of the cluster-tilting object.

The results of Section~\ref{section.2apr} have been proven without assuming that the algebra is representation finite. Hence, with the alterations mentioned above, the algorithm works as presented in Section~\ref{section_algorithm} in this more general setup. \\

The change to the definition of homotopy of local slices is fairly straight forward.

\begin{definition} \label{definition.homotopy_inf}
Let $T$ be a cluster-tilting object in $\cC$, and $B = \End_{\cC}(T)^{\op}$. Let $\Sigma$ and $\Sigma'$ be two local slices in $\mod B$. We say that $\Sigma$ and $\Sigma'$ are homotopic, if $\Sigma \sim \Sigma'$ in the sense of Definition~\ref{definition.homotopy}, or if $\mathcal{C}$ is infinite, $T$ has no regular direct summands, and \emph{all} direct summands of $\tau T$ in the connecting component lie either at the same side of both $\Sigma$ and $\Sigma'$ or in between them.
\end{definition}

\begin{remark} Note that the sufficiency part of Theorem~\ref{theorem.homotopy} is also valid for this definition of ``homotopic''. However, for the necessity part it remains to deal with the case when the cluster-tilting object has nonzero regular summands. Using the same notation as in Theorem~\ref{theorem.homotopy}, assume $T$ has nonzero regular summands and pick two local slices $\Sigma \sim \Sigma'$ in $\mod B$. The critical case is when either all summands of $\tau T$ lie at the same side of both local slices or in between them. In any case, they both kill the same arrows from $Q_B$. Thus the theorem remains valid under this setting.
\end{remark}

Now we change the definition of the equivalence relation $\equiv$, and hence of cells and trenches, to fit this more general setup.

\begin{definition}\label{definition.celldecomp.infinity}
Let $H$ be a hereditary algebra and $T \in \mod H$ a tilting module. For $T'$ and $T''$ indecomposable summands of $T$ we write $T' \ors T''$ if at least one of
\[ \Hom_H(T', T'') \neq 0 \quad \vee \quad \Hom_H(\tau T', T'') \neq 0 \quad \vee \quad \Hom_H(T', \tau^- T'') \neq 0 \] holds.
We denote by $\orr$ the transitive hull of this relation. We write $T' \eqr T''$ if $T'$ and $T''$ are both regular, or $T' \orr T'' \orr T'$. This is an equivalence relation. Note that $\orr$ induces a partial order on the equivalence classes.

We use similar constructions in $\mathcal{D} = D^b(\mod H)$.

For $T \in \cC$ cluster-tilting, and $\Sigma$ a local slice with $\Sigma \cap \add \tau T=0$, we use the corresponding tilting module $D\Hom_{\cC}(T, \Sigma)$ over the hereditary algebra $\End_{\cC}(\Sigma)^{\op}$ to obtain similar notions. For $T'$ and $T''$ indecomposable summands of $T$ we write $T' \ors_{\Sigma} T''$ if $D \Hom_{\mathcal{C}}(T', \Sigma) \ors D \Hom_{\mathcal{C}}(T'', \Sigma)$. Similarly we obtain an equivalence relation $\eqr_{\Sigma}$.
\end{definition}

\begin{remarks} \mbox{}
\begin{enumerate}
\item It appears as if our definition of the equivalence relation $\eqr_{\Sigma}$ in $\cC$ depends on the choice of $\Sigma$. We will see that this is not the case (see Corollary~\ref{corollary.independant}).
\item Note that the set of complete slices in $\mod H$ forms a lattice (i.e.\ is partially ordered and has suprema and infima -- this is induced by comparing $\tau$-orbit-wise).
\end{enumerate}
\end{remarks}

Next we prove some technical lemmas which will be useful for the rest of the section.
\begin{lemma} \label{lemma.notintauleft}
Let $T_1 \oplus T_2$ be a tilting module over a hereditary algebra $H$, and $\Sigma$ the smallest complete slice containing $T_1$. Then $\add T_2 \cap \tau \Sigma = 0$.
\end{lemma}

\begin{proof}
 Assume that $0\ne T' \in \add T_2 \cap \tau \Sigma$. Since $\Sigma$ is the smallest complete slice containing $T_1$, there is a non-zero morphism $T' \to \tau T_1$. This means that $ 0 \ne \Hom_H (T',\tau T_1) = D\Ext^1_H(T_1,T') $, contradicting the fact that $T$ is a tilting module.
\end{proof}

\begin{lemma} \label{lemma.insliced}
Let $T$ be a tilting module over a hereditary algebra $H$, $T' \in \add T$ indecomposable non-regular. Then the equivalence class $\cl{T'}$ is contained in a complete slice.
\end{lemma}

\begin{proof}
Since $T'$ is non-regular, it is contained in some complete slice. Let $S$ be maximal such that $\{T'\} \subseteq S \subseteq \cl{T'}$, with $S$ contained in a complete slice. Assume $T'' \in \cl{T'} \setminus S$. By Lemma~\ref{lemma.notintauleft} and the definition of $\orr$, the object $T''$ cannot lie properly to the left of the minimal slice containing $S$, and dually it cannot lie to the right of the maximal slice containing $S$. Hence there is a complete slice containing $S$ and $T''$.
\end{proof}

The following proposition follows immediately from Lemma~\ref{lemma.insliced}, looking at the projection $\cD \to \cC$.

\begin{proposition} \label{proposition.inslice}
Let $T$ be cluster-tilting in $\mathcal{C}$, and $\Sigma$ some local slice with $\Sigma \cap \add \tau T = 0$. Let $T' \in add T$ be indecomposable non-regular. Then $\cl[\Sigma]{T'}$ is contained in some local slice. 
\end{proposition}

The next corollary shows that summands of a cluster-tilting object, which are equivalent with respect to $\eqr_{\Sigma}$, can be lifted to the derived category in such a way that they remain equivalent.
\begin{corollary} \label{corollary.liftequiv}
Let $T$ be cluster-tilting in $\mathcal{C}$, and $\Sigma$ some local slice with $\Sigma \cap \add \tau T = 0$. Let $T_0 \in add T$ be indecomposable non-regular, and $\cl[\Sigma]{T_0} = \{T_0, \ldots, T_r\}$. Then we can find preimages $\{T_0^{\mathcal{D}}, \ldots, T_r^{\mathcal{D}} \}$ in $\mathcal{D}$ which lie in one complete slice.

In particular
\begin{align*}
\Hom_{\mathcal{D}}(T_i^{\mathcal{D}}, T_j^{\mathcal{D}}) & = \Hom_{\mathcal{C}}(T_i, T_j),\\
\Hom_{\mathcal{D}}(\tau T_i^{\mathcal{D}}, T_j^{\mathcal{D}}) & = \Hom_{\mathcal{C}}(\tau T_i, T_j), \text{ and}\\
\Hom_{\mathcal{D}}(T_i^{\mathcal{D}}, \tau^- T_j^{\mathcal{D}}) & = \Hom_{\mathcal{C}}(T_i, \tau^- T_j),
\end{align*}
and $T_i^{\mathcal{D}} \orr T_j^{\mathcal{D}}$ for any $i, j$.
\end{corollary}
In order to obtain all maximal tilted subalgebras, we must make sure that there is no slice cutting through our cells.
\begin{lemma} \label{lemma.nocellcuttingd}
Let $S$ be a subset of some complete slice in $\mathcal{D}$, such that $T' \orr T''$ for any $T', T'' \in S$. Let $\Sigma$ be a complete slice with $\Sigma \cap \tau S = 0$. Then either all of $\tau S$ lie to the left or all of $\tau S$ lie to the right of $\Sigma$.
\end{lemma}

\begin{proof}
We may assume that $S$ has some element which lies to the right of $\tau^- \Sigma$. By definition of $\orr$ and the fact that $S \cap \tau^- \Sigma = 0$, then all elements of $S$ lie to the right of $\tau^- \Sigma$.
\end{proof}

Next we show that going down from the derived category to the cluster category is compatible with our equivalences.

\begin{lemma} \label{lemma.equivindep}
Let $S$ be a subset of some complete slice in $\mathcal{D}$, such that $T' \orr T''$ for any $T', T'' \in S$. Let $\Sigma$ be a local slice in $\mathcal{C}$ with $\Sigma \cap \tau \mathrm{pr}(S) = 0$ (here $\mathrm{pr} \colon \mathcal{D} \to \mathcal{C}$ is the projection functor). Then $\mathrm{pr}(T') \orr_{\Sigma} \mathrm{pr}(T'')$ for any $T', T'' \in S$.
\end{lemma}

\begin{proof}
Assume $T' \ors T''$, but $\mathrm{pr}(T') \not\ors_{\Sigma} \mathrm{pr}(T'')$. If $\Hom_{\mathcal{D}}(T', T'') \neq 0$, then by Lemma~\ref{lemma.nocellcuttingd} we have $\Hom_{\mathcal{C} / (\tau^- \Sigma)}(\mathrm{pr}(T'), \mathrm{pr}(T'')) \neq 0$, contradicting our assumption. Hence we may assume $\Hom_{\mathcal{D}}(\tau T', T'') \neq 0$, and any map $\tau T' \to T''$ factors through $\tau^- \Sigma'$ for some complete slice $\Sigma'$ with $\mathrm{pr}(\Sigma') = \Sigma$. By Lemma~\ref{lemma.nocellcuttingd} all of $S$ lies to the right of $\tau^- \Sigma'$, and $\tau T' \in \tau^- \Sigma'$. Similarly $\Hom_{\mathcal{C} / (\tau^- \Sigma)}(\mathrm{pr}(T'),\mathrm{pr}( \tau^- T'')) = 0$ implies $\tau^- T'' \in \tau^- F \Sigma'$. But then $T'' \in \tau^{-}\Sigma'[1]$, and hence $\Hom_{\mathcal{D}}(\tau T', T'') = 0$, contradicting our assumption.
\end{proof}
We now have all ingredients needed to prove that the definition of the equivalence relation $\eqr_{\Sigma}$ is independent of the chosen local slice $\Sigma$.
\begin{corollary} \label{corollary.independant}
Let $T$ be cluster-tilting in $\mathcal{C}$. Then $\eqr_{\Sigma}$ is independent of the choice of local slice $\Sigma$ with $\Sigma \cap \add \tau T = 0$.
\end{corollary}

We will therefore from now on only write $\eqr$.

\begin{proof}
We have seen in Corollary~\ref{corollary.liftequiv} that summands equivalent with respect to one slice can be lifted to equivalent objects in the derived category. Then by Lemma~\ref{lemma.equivindep} they are also equivalent with respect to any other local slice.
\end{proof}

Now one makes sure that maps inside the cell are not affected by the choice of local slice.
\begin{proposition}
Let $T$ be cluster-tilting in $\mathcal{C}$. For any local slice $\Sigma$ with $\Sigma \cap \add \tau T = 0$ and any $T' \eqr T''$ we have
\[ \Hom_H( \widetilde{T}', \widetilde{T}'') = \left\{ \begin{array}{ll} \Hom_{\mathcal{C}}(T', T'') & \text{if } T' \text{ non-regular} \\ \frac{\Hom_{\mathcal{C}}(T', T'')}{\left(\substack{\text{maps factoring through}\\ \text{non-regular objects}}\right)} &  \text{if } T' \text{ regular,} \end{array} \right. \]
where $H = \End_{\mathcal{C}}(\Sigma)^{\op}$, $\wT' = D\Hom_{\cC}(T',\Sigma)$, and $\wT'' = D\Hom_{\cC}(T'',\Sigma)$.

In particular it is independent of $\Sigma$.
\end{proposition}

\begin{proof}
Note that 
\[ \Hom_H( \widetilde{T}', \widetilde{T}'') = \frac{\Hom_{\mathcal{C}}(T', T'')}{\left(\substack{\text{maps factoring}\\ \text{through } \tau^- \Sigma} \right)}. \]
The claim for $T'$ and $T''$ regular follows immediately, since any map between them factors through the non-regular component if and only if it factors through any local slice.

For  $T'$ and $T''$ non-regular the claim follows from Corollary~\ref{corollary.liftequiv} and Lemma~\ref{lemma.nocellcuttingd}.
\end{proof}

Let $Q$ be a tree-quiver (that is a quiver without cycles, but possibly with multiple edges) and $\cC$ the cluster category of the path algebra $\mathbb{K}Q$. For any cluster-tilting object $T$ in $\cC$, we say that $\End_{\cC}(T)^\op$ is a \emph{cluster-tilted algebra of tree-type} (\cite[\S 4]{ABS2}). The next proposition assures that Definition~\ref{definition.celldecomp} and  Definition~\ref{definition.celldecomp.infinity} are equivalent for cluster-tilted algebras of tree-type.

\begin{proposition}
Assume $\mathcal{C}$ of tree-type, $T$ cluster-tilting, and $T', T''$ indecomposable non-regular. Then $T' \eqr T''$ if and only if $T' \equiv T''$, with $\equiv$ as defined in Section~\ref{section.Loc_slices}.
\end{proposition}

\begin{proof}
It is easy to see that $T' \equiv T''$ implies $T' \eqr T''$.

For the converse, note that since $\mathcal{C}$ is of tree-type, so is any local slice. In particular, by Proposition~\ref{proposition.inslice}, the set $\cl{T'}$ is contained in a local slice of tree-type. We may assume that $T' \eqr T''$, and there is no element of $\cl{T'}$ in this local slice between them. It is easy to see that this can only happen if $T' \oplus T'' \in \add(\tau X \oplus \vartheta X \oplus X)$ for some AR-triangle $\tau X \to \vartheta X \to X \to$ in $\mathcal{C}$. Hence $T' \equiv T''$.
\end{proof}

Now we are ready to jump trenches.

Assume $T \in \mathcal{C}$ is cluster-tilting and $\Sigma$ a local slice with $\Sigma \cap \add \tau T = 0$. There are two different cases:
\begin{enumerate}
\item There are no summands of $T$ in the connecting component to the right of $\tau^- \Sigma$. In case there are also no regular direct summands of $T$, the local slice $\Sigma$ is homotopic to any local slice $\Sigma'$ such that there are no direct summands of $T$ left of (or in) $\tau^- \Sigma'$ (see Definition~\ref{definition.homotopy_inf}). We proceed using this local slice.

In case there is at least one regular direct summand $T'$ of $T$ the trench $\tau \cl{T'} = \{\tau T' \mid T' \text{ regular summand of } T\}$ is the one to jump. That is, we also replace $\Sigma$ by $\Sigma'$ as above, but they are not homotopic, and, on the level of tilted algebras, we apply the generalized $2$-APR tilt associated with $\cl{T'}$.
\item There is some direct summand of $T$ finitely many steps to the right of $\Sigma$ (equivalently, $D\Hom_{\mathcal{C}}(T, \Sigma)$ has a preprojective direct summand). We may assume $\Sigma$ to be as far to the right as possible inside its homotopy class. Then any source of $\Sigma$ is of the form $\tau^2 T'$ for some $T' \in \add T$. For any $X \in \tau^{-3} \Sigma$ there is a non-zero map $T \to \tau X$, and hence $X \not\in \add T$. In particular, any equivalence class $\cl{T'}$ with $T' \in \tau^{-2} \Sigma$ must be contained in $\tau^{-2} \Sigma$. Among these classes, choose $\cl{T'}$ minimal with respect to $\orr_{\Sigma}$. We let $\widetilde{\Sigma}$ be a local slice containing $\cl{T'}$, and such that all sinks of $\widetilde{\Sigma}$ lie in $\cl{T'}$. Choose the slice $\Sigma'$ $\tau$-orbit wise by
\[ \Sigma'_o = \left\{ \begin{array}{ll} \widetilde{\Sigma}_o & \text{ if } \widetilde{\Sigma}_o \in \{\tau^- \Sigma_o, \tau^{-2} \Sigma_o \} \\ \Sigma_o & \text{ otherwise} \end{array} \right. \]
That is, we take $\Sigma' = \widetilde{\Sigma}$ if the slices $\Sigma$ and $\widetilde{\Sigma}$ don't intersect, and otherwise we choose the rightmost points of $\Sigma$ and $\widetilde{\Sigma}$ $\tau$-orbit wise.

We now check that $\Sigma'$ is a ``legal'' slice, that is that $\Sigma' \cap \add \tau T = 0$. Assume $\Sigma' \cap \add \tau T \neq 0$, say $\tau T'' \in \Sigma'$. Then clearly $\tau T'' \in \widetilde{\Sigma}$, and $T'' \in \tau^{-2} \Sigma$. By the first property we have $\Hom(\tau T'', \cl{T'}) \neq 0$, which, together with the second property, contradicts the minimality in our choice of $T'$. Hence $\Sigma' \cap \add \tau T = 0$.

Now clearly replacing $\Sigma$ by $\Sigma'$ jumps the trench $\tau \cl{T'}$, and it remains to see that no other trenches are jumped. Let $T'' \in \add T \setminus \cl{T'}$ be indecomposable. If $T'' \not\in \tau^{-2} \Sigma$ then the trench $\tau \cl{T''}$ cannot be affected by our jump. If $T'' \in \tau^{-2} \Sigma$, then, by choice of $T'$, we have $\Hom(T'', T') = 0$. Hence $T''$ is not in $\widetilde{\Sigma}$, and thus $\tau T''$ is not in $\tau \widetilde{\Sigma}$. But the space between $\Sigma$ and $\Sigma'$ is contained in $\tau \widetilde{\Sigma}$, hence the trench $\tau \cl{T''}$ cannot have been jumped. 
\end{enumerate}

We illustrate the procedure above with an example.

\begin{example} \label{example.inf1}
Let $B$ be the cluster-tilted algebra of type
$\begin{tikzpicture}[yscale=-1]
 \node (1) at (0,0) [vertex] {};
 \node (2) at (.71,0) [vertex] {};
 \node (3) at (1.42,0) [vertex] {};
 \draw [->,bend right=20] (1) to (2);
 \draw [->,bend left=20] (1) to (2);
 \draw [->,bend right=20] (2) to (3);
 \draw [->,bend left=20] (2) to (3);
\end{tikzpicture}$
with quiver as depicted below.
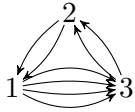
\[ \begin{tikzpicture}[yscale=-1]
 \node (1) at (0,0) [inner sep=1pt] {$1$};
 \node (3) at (1.5,0) [inner sep=1pt] {$3$};
 \node (2) at (.75,-1) [inner sep=1pt] {$2$};
 \draw [->,bend right=30] (1) to (3);
 \draw [->,bend right=10] (1) to (3);
 \draw [->,bend left=10] (1) to (3);
 \draw [->,bend left=30] (1) to (3);
 \draw [->,bend right=15] (2) to (1);
 \draw [->,bend left=15] (2) to (1);
 \draw [->,bend right=15] (3) to (2);
 \draw [->,bend left=15] (3) to (2);
\end{tikzpicture} \]
(This is obtained from the hereditary algebra by mutating at the center vertex.) Then the homotopy classes of local slices look as depicted in Figure~\ref{figure.inf1}. We see that there are three maximal tilted subalgebras:
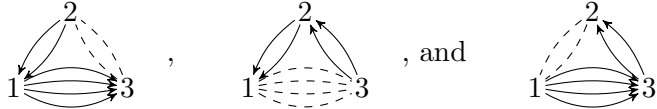
\[ \begin{tikzpicture}[yscale=-1,baseline=10pt]
 \node (1) at (0,0) [inner sep=1pt] {$1$};
 \node (3) at (1.5,0) [inner sep=1pt] {$3$};
 \node (2) at (.75,-1) [inner sep=1pt] {$2$};
 \draw [->,bend right=30] (1) to (3);
 \draw [->,bend right=10] (1) to (3);
 \draw [->,bend left=10] (1) to (3);
 \draw [->,bend left=30] (1) to (3);
 \draw [->,bend right=15] (2) to (1);
 \draw [->,bend left=15] (2) to (1);
 \draw [dashed,bend right=15] (3) to (2);
 \draw [dashed,bend left=15] (3) to (2);
\end{tikzpicture} \quad, \qquad
\begin{tikzpicture}[yscale=-1,baseline=10pt]
 \node (1) at (0,0) [inner sep=1pt] {$1$};
 \node (3) at (1.5,0) [inner sep=1pt] {$3$};
 \node (2) at (.75,-1) [inner sep=1pt] {$2$};
 \draw [dashed,bend right=30] (1) to (3);
 \draw [dashed,bend right=10] (1) to (3);
 \draw [dashed,bend left=10] (1) to (3);
 \draw [dashed,bend left=30] (1) to (3);
 \draw [->,bend right=15] (2) to (1);
 \draw [->,bend left=15] (2) to (1);
 \draw [->,bend right=15] (3) to (2);
 \draw [->,bend left=15] (3) to (2);
\end{tikzpicture} \quad \text{, and} \qquad
\begin{tikzpicture}[yscale=-1,baseline=10pt]
 \node (1) at (0,0) [inner sep=1pt] {$1$};
 \node (3) at (1.5,0) [inner sep=1pt] {$3$};
 \node (2) at (.75,-1) [inner sep=1pt] {$2$};
 \draw [->,bend right=30] (1) to (3);
 \draw [->,bend right=10] (1) to (3);
 \draw [->,bend left=10] (1) to (3);
 \draw [->,bend left=30] (1) to (3);
 \draw [dashed,bend right=15] (2) to (1);
 \draw [dashed,bend left=15] (2) to (1);
 \draw [->,bend right=15] (3) to (2);
 \draw [->,bend left=15] (3) to (2);
\end{tikzpicture} \]
\end{example}

\begin{figure}[htb]
\[ \begin{tikzpicture}[scale=.8,yscale=-1,baseline]
 \fill [fill1] (4,-.5) -- (2,2.5) -- (-.5,2.5) -- (-.5,-.5) -- cycle;
 \fill [fill2] (4,-.5) -- (4.5,0) [bend right=20] to (3.25,1.25) [bend right=20] to (2,2.5) -- (1.5,2) [bend right=20] to (2.75,.75) [bend right=20] to (4,-.5);
 \foreach \x in {0,...,5}
  {
   \node (1-\x) at (2*\x,2) [vertex] {};
   \node (3-\x) at (2*\x,0) [vertex] {};
  }
 \foreach \x in {0,...,4}
  \node (2-\x) at (2*\x+1,1) [vertex] {};
 \replacevertex{(1-2)}{[tvertex] {$\tau T_3$}}
 \replacevertex{(1-3)}{[tvertex] {$T_3$}}
 \replacevertex{(3-3)}{[tvertex] {$\tau T_1$}}
 \replacevertex{(3-4)}{[tvertex] {$T_1$}}
 \foreach \x/\y in {0/1,1/2,2/3,3/4,4/5}
  {
   \draw [->,bend right=20] (1-\x) to (2-\x);
   \draw [->,bend left=20] (1-\x) to (2-\x);
   \draw [->,bend right=20] (3-\x) to (2-\x);
   \draw [->,bend left=20] (3-\x) to (2-\x);
   \draw [->,bend right=20] (2-\x) to (1-\y);
   \draw [->,bend left=20] (2-\x) to (1-\y);
   \draw [->,bend right=20] (2-\x) to (3-\y);
   \draw [->,bend left=20] (2-\x) to (3-\y);
  }
 \node at (-.5,0) {$\cdots$};
 \node at (-.5,1) {$\cdots$};
 \node at (-.5,2) {$\cdots$};
 \node at (10.5,0) {$\cdots$};
 \node at (10.5,1) {$\cdots$};
 \node at (10.5,2) {$\cdots$};
\end{tikzpicture}
\quad
\begin{tikzpicture}[yscale=-1,baseline=40pt]
 \node (1) at (0,0) [inner sep=1pt] {$1$};
 \node (3) at (1.5,0) [inner sep=1pt] {$3$};
 \node (2) at (.75,-1) [inner sep=1pt] {$2$};
 \draw [->,bend right=30] (1) to (3);
 \draw [->,bend right=10] (1) to (3);
 \draw [->,bend left=10] (1) to (3);
 \draw [->,bend left=30] (1) to (3);
 \draw [->,bend right=15] (2) to (1);
 \draw [->,bend left=15] (2) to (1);
 \draw [dashed,bend right=15] (3) to (2);
 \draw [dashed,bend left=15] (3) to (2);
\end{tikzpicture} \]
\[ \begin{tikzpicture}[scale=.8,yscale=-1,baseline]
 \fill [fill2] (4,-.5) -- (3.5,0) [bend left=20] to (4.75,1.25) [bend left=20] to (6,2.5) -- (6.5,2) [bend left=20] to (5.25,.75) [bend left=20] to (4,-.5);
 \foreach \x in {0,...,5}
  {
   \node (1-\x) at (2*\x,2) [vertex] {};
   \node (3-\x) at (2*\x,0) [vertex] {};
  }
 \foreach \x in {0,...,4}
  \node (2-\x) at (2*\x+1,1) [vertex] {};
 \replacevertex{(1-2)}{[tvertex] {$\tau T_3$}}
 \replacevertex[fill2]{(1-3)}{[tvertex] {$T_3$}}
 \replacevertex{(3-3)}{[tvertex] {$\tau T_1$}}
 \replacevertex{(3-4)}{[tvertex] {$T_1$}}
 \foreach \x/\y in {0/1,1/2,2/3,3/4,4/5}
  {
   \draw [->,bend right=20] (1-\x) to (2-\x);
   \draw [->,bend left=20] (1-\x) to (2-\x);
   \draw [->,bend right=20] (3-\x) to (2-\x);
   \draw [->,bend left=20] (3-\x) to (2-\x);
   \draw [->,bend right=20] (2-\x) to (1-\y);
   \draw [->,bend left=20] (2-\x) to (1-\y);
   \draw [->,bend right=20] (2-\x) to (3-\y);
   \draw [->,bend left=20] (2-\x) to (3-\y);
  }
 \node at (-.5,0) {$\cdots$};
 \node at (-.5,1) {$\cdots$};
 \node at (-.5,2) {$\cdots$};
 \node at (10.5,0) {$\cdots$};
 \node at (10.5,1) {$\cdots$};
 \node at (10.5,2) {$\cdots$};
\end{tikzpicture}
\quad
\begin{tikzpicture}[yscale=-1,baseline=40pt]
 \node (1) at (0,0) [inner sep=1pt] {$1$};
 \node (3) at (1.5,0) [inner sep=1pt] {$3$};
 \node (2) at (.75,-1) [inner sep=1pt] {$2$};
 \draw [dashed,bend right=30] (1) to (3);
 \draw [dashed,bend right=10] (1) to (3);
 \draw [dashed,bend left=10] (1) to (3);
 \draw [dashed,bend left=30] (1) to (3);
 \draw [->,bend right=15] (2) to (1);
 \draw [->,bend left=15] (2) to (1);
 \draw [->,bend right=15] (3) to (2);
 \draw [->,bend left=15] (3) to (2);
\end{tikzpicture} \]
\[ \begin{tikzpicture}[scale=.8,yscale=-1,baseline]
 \fill [fill1] (8,-.5) -- (6,2.5) -- (10.5,2.5) -- (10.5,-.5) -- cycle;
 \fill [fill2] (8,-.5) -- (8.5,0) [bend right=20] to (7.25,1.25) [bend right=20] to (6,2.5) -- (5.5,2) [bend right=20] to (6.75,.75) [bend right=20] to (8,-.5);
 \foreach \x in {0,...,5}
  {
   \node (1-\x) at (2*\x,2) [vertex] {};
   \node (3-\x) at (2*\x,0) [vertex] {};
  }
 \foreach \x in {0,...,4}
  \node (2-\x) at (2*\x+1,1) [vertex] {};
 \replacevertex{(1-2)}{[tvertex] {$\tau T_3$}}
 \replacevertex[fill2]{(1-3)}{[tvertex] {$T_3$}}
 \replacevertex{(3-3)}{[tvertex] {$\tau T_1$}}
 \replacevertex[fill2]{(3-4)}{[tvertex] {$T_1$}}
 \foreach \x/\y in {0/1,1/2,2/3,3/4,4/5}
  {
   \draw [->,bend right=20] (1-\x) to (2-\x);
   \draw [->,bend left=20] (1-\x) to (2-\x);
   \draw [->,bend right=20] (3-\x) to (2-\x);
   \draw [->,bend left=20] (3-\x) to (2-\x);
   \draw [->,bend right=20] (2-\x) to (1-\y);
   \draw [->,bend left=20] (2-\x) to (1-\y);
   \draw [->,bend right=20] (2-\x) to (3-\y);
   \draw [->,bend left=20] (2-\x) to (3-\y);
  }
 \node at (-.5,0) {$\cdots$};
 \node at (-.5,1) {$\cdots$};
 \node at (-.5,2) {$\cdots$};
 \node at (10.5,0) {$\cdots$};
 \node at (10.5,1) {$\cdots$};
 \node at (10.5,2) {$\cdots$};
\end{tikzpicture}
\quad
\begin{tikzpicture}[yscale=-1,baseline=40pt]
 \node (1) at (0,0) [inner sep=1pt] {$1$};
 \node (3) at (1.5,0) [inner sep=1pt] {$3$};
 \node (2) at (.75,-1) [inner sep=1pt] {$2$};
 \draw [->,bend right=30] (1) to (3);
 \draw [->,bend right=10] (1) to (3);
 \draw [->,bend left=10] (1) to (3);
 \draw [->,bend left=30] (1) to (3);
 \draw [dashed,bend right=15] (2) to (1);
 \draw [dashed,bend left=15] (2) to (1);
 \draw [->,bend right=15] (3) to (2);
 \draw [->,bend left=15] (3) to (2);
\end{tikzpicture} \]
\caption{The two homotopy classes of local slices for the cluster-tilted algebra of Example~\ref{example.inf1}. The light grey areas mark the homotopy classes of local slices, the darker grey marks the rightmost (upper picture), only (middle picture), and leftmost (lower picture) local slices in their homotopy class. Note that $T_2$ is regular, and hence not to be seen in the picture of the connecting component.} \label{figure.inf1}
\end{figure}
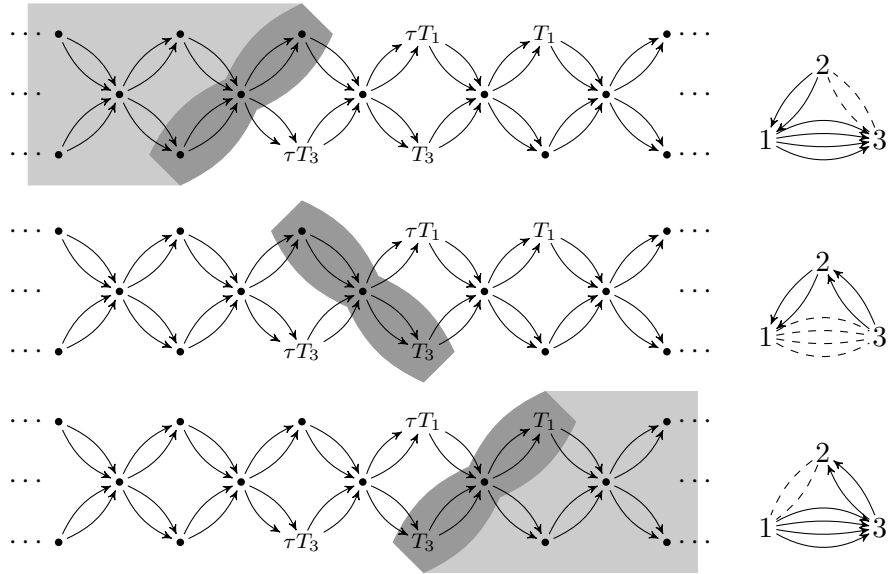

\begin{remark}
Note that in the representation infinite case there could be less maximal tilted subalgebras than one might expect. In case $B$ is the endomorphism ring of a regular cluster-tilting object it only has one maximal tilted subalgebra.
\end{remark}


\newpage

\begin{thebibliography}{BMR{\etalchar{+}}1}

\bibitem[A1]{CC}
Claire Amiot.
\newblock {C}luster categories for algebras of global dimension 2 and quivers
  with potential.
\newblock Preprint, to appear in Ann.\ Inst.\ Fourier (Grenoble), arXiv:0805.1035v2.

\bibitem[A2]{C_PhD}
Claire Amiot.
\newblock { Sur les petites cat\'{e}gories triangul\'{e}es}.
\newblock {\em Bull. Soc. Math. France}, 135, no. 3, 435--474, 2007.

\bibitem[ABS1]{ABS1}
I.~Assem, T.~Br{\"u}stle, and R.~Schiffler.
\newblock Cluster-tilted algebras as trivial extensions.
\newblock {\em Bull. Lond. Math. Soc.}, 40(1):151--162, 2008.

\bibitem[ABS2]{ABS2}
Ibrahim Assem, Thomas Br{\"u}stle, and Ralf Schiffler.
\newblock Cluster-tilted algebras and slices.
\newblock {\em J. Algebra}, 319(8):3464--3479, 2008.

\bibitem[APR]{APR}
Maurice Auslander, Mar{\'{\i}}a~In{\'e}s Platzeck, and Idun Reiten.
\newblock Coxeter functors without diagrams.
\newblock {\em Trans. Amer. Math. Soc.}, 250:1--46, 1979.

\bibitem[AR]{AR}
Maurice Auslander and Idun Reiten.
\newblock Representation theory of {A}rtin algebras. {III}. {A}lmost split
  sequences.
\newblock {\em Comm. Algebra}, 3:239--294, 1975.

\bibitem[BFPPT]{BFPPT}
Michael Barot, Elsa Fern{\'a}ndez, Mar{\'{\i}}a~In{\'e}s Platzeck, Nilda~Isabel
  Pratti, and Sonia Trepode.
\newblock From iterated tilted algebras to cluster-tilted algebras.
\newblock Preprint, arXiv:0811.1328v1.

\bibitem[B{\O}OW]{BOW}
Marco~A. Bertani-{\O}kland, Steffen Opperman and Anette Wr{\aa}lsen.
\newblock Finding a cluster-tilting object for a representation finite cluster-tilted algebra.
\newblock Preprint, arXiv:0912.2911v1.

\bibitem[BFT]{BFT}
Natalia Bordino, Elsa Fern\'{a}ndez, and Sonia Trepode.
\newblock On cluster tilted algebras arising from strongly simply connected algebras.
\newblock In preparation.

\bibitem[BIRSm]{BIRSm}
Aslak~B. Buan, Osamu Iyama, Idun Reiten, and David Smith.
\newblock Mutation of cluster-tilting objects and potentials.
\newblock Preprint, arXiv:0804.3813v3.

\bibitem[BMR1]{BMR1}
Aslak~Bakke Buan, Robert~J. Marsh, and Idun Reiten.
\newblock Cluster-tilted algebras.
\newblock {\em Trans. Amer. Math. Soc.}, 359(1):323--332 (electronic), 2007.

\bibitem[BMR2]{BMR2}
Aslak~Bakke Buan, Robert~J. Marsh, and Idun Reiten.
\newblock Cluster mutation via quiver representations.
\newblock {\em Comment. Math. Helv.}, 83(1):143--177, 2008.

\bibitem[BMRRT]{BMRRT}
Aslak~Bakke Buan, Robert Marsh, Markus Reineke, Idun Reiten, and Gordana
  Todorov.
\newblock Tilting theory and cluster combinatorics.
\newblock {\em Adv. Math.}, 204(2):572--618, 2006.

\bibitem[BR]{BR}
Aslak~Bakke Buan and Idun Reiten.
\newblock {F}rom tilted to cluster-tilted algebras of {D}ynkin type.
\newblock Preprint, math.RT/0510445v1.

\bibitem[BRS]{BRS}
Aslak~Bakke Buan, Idun Reiten, and Ahmet~I. Seven.
\newblock Tame concealed algebras and cluster quivers of minimal infinite type.
\newblock {\em J. Pure Appl. Algebra}, 211(1):71--82, 2007.

\bibitem[BV]{BV}
Aslak~Bakke Buan and Dagfinn~F. Vatne.
\newblock Derived equivalence classification for cluster-tilted algebras of
  type {$A\sb n$}.
\newblock {\em J. Algebra}, 319(7):2723--2738, 2008.

\bibitem[CCS]{otherCCS}
Philippe Caldero, Fr{\'e}d{\'e}ric Chapoton, and Ralf Schiffler.
\newblock Quivers with relations and cluster tilted algebras.
\newblock {\em Algebr. Represent. Theory}, 9(4):359--376, 2006.

\bibitem[FZ]{FZ}
Sergey Fomin and Andrei Zelevinsky.
\newblock Cluster algebras. {I}. {F}oundations.
\newblock {\em J. Amer. Math. Soc.}, 15(2):497--529 (electronic), 2002.

\bibitem[Ha]{Happel}
Dieter Happel.
\newblock {\em Triangulated categories in the representation theory of
  finite-dimensional algebras}, volume 119 of {\em London Mathematical Society
  Lecture Note Series}.
\newblock Cambridge University Press, Cambridge, 1988.

\bibitem[HR]{HR}
Dieter Happel and Claus M. Ringel.
\newblock Tilted algebras.
\newblock {\em Trans. Amer. Math. Soc.} 274 (1982) , no. 2, 399-443.

\bibitem[HRS]{HRS}
Dieter Happel, Jeremy Rickard, and Aidan Schofield.
\newblock Piecewise hereditary algebras.
\newblock {\em Bull. London Math. Soc.}, 20(1):23--28, 1988.

\bibitem[Hu]{Hubner} Thomas H\"ubner.
\newblock Rank additivity for Quasi-tilted algebras of canonical type.
\newblock {\em Colloquium Mathematicum}, 75 (1998), no. 2.

\bibitem[IO]{IO}
Osamu Iyama and Steffen Oppermann.
\newblock n-representation-finite algebras and n-APR tilting.
\newblock Preprint, arXiv:0909.0593v1.

\bibitem[K]{K} 
Bernhard Keller and Michel Van Den Bergh.
\newblock Deformed Calabi-Yau completions.
\newblock Preprint, arXiv:0908.3499v5.

\bibitem[MRZ]{MRZ}
Robert Marsh, Markus Reineke, and Andrei Zelevinsky.
\newblock Generalized associahedra via quiver representations.
\newblock {\em Trans. Amer. Math. Soc.}, 355(10):4171--4186 (electronic), 2003.

\bibitem[S]{Schiffler}
Ralf Schiffler.
\newblock A geometric model for cluster categories of type {$D_{n}$}.
\newblock {\em J. Algebraic Combin.}, 27(1):1--21, 2008.

\bibitem[V]{Dagfinn}
Dagfinn Vatne.
\newblock The mutation class of {$D_n$} quivers.
\newblock Preprint, to appear in {\em Comm. Algebra}, arXiv:0810.4789v1.

\bibitem[W]{Min}
Anette Wr{\aa}lsen.
\newblock Rigid objects in higher cluster categories.
\newblock {\em J. Algebra}, 321(2):532--547, 2009.

\bibitem[Z]{Z}
Bin Zhu
\newblock Equivalences between cluster categories.
\newblock {\em J. Algebra}, 304(2):832--850, 2006.

\end{thebibliography}
\end{document}